\documentclass{imsart}

\RequirePackage{amsthm,amsmath,amsfonts,amssymb,enumitem,aliascnt}
\RequirePackage{MnSymbol}
\RequirePackage[dvipsnames]{xcolor}
\RequirePackage[numbers]{natbib}
\RequirePackage[colorlinks,citecolor=blue,urlcolor=blue]{hyperref}%% uncomment this for coloring bibliography citations and linked URLs
\RequirePackage{graphicx}%% uncomment this for including figures

\usepackage[left=3cm, right=3cm, top=3cm, bottom=3cm]{geometry}

\startlocaldefs
\allowdisplaybreaks

\theoremstyle{plain}

\newtheorem{theorem}{Theorem}

\newaliascnt{corollary}{theorem}

\aliascntresetthe{corollary}

\newaliascnt{lemma}{theorem}
\newtheorem{lemma}[lemma]{Lemma}
\aliascntresetthe{lemma}

\newaliascnt{proposition}{theorem}

\aliascntresetthe{proposition}

\theoremstyle{definition}
\newaliascnt{definition}{theorem}
\newtheorem{definition}[definition]{Definition}
\aliascntresetthe{definition}

\theoremstyle{remark}
\newaliascnt{remark}{theorem}

\aliascntresetthe{remark}

\renewcommand{\mid}{\,\middle|\,}

\renewcommand{\P}{\text{P}}
\newcommand{\E}{\text{E}}
\newcommand{\Cov}{\text{Cov}}
\newcommand{\Var}{\text{Var}}
\newcommand{\bSigma}{\boldsymbol{\Sigma}}

\newcommand{\bG}{\mathbf{G}}
\newcommand{\bH}{\mathbf{H}}
\newcommand{\bR}{\mathbf{R}}

\newcommand{\bT}{\mathbf{T}}
\newcommand{\bU}{\mathbf{U}}

\newcommand{\bW}{\mathbf{W}}
\newcommand{\bX}{\mathbf{X}}
\newcommand{\bY}{\mathbf{Y}}
\newcommand{\bZ}{\mathbf{Z}}
\newcommand{\be}{\mathbf{e}}

\newcommand{\bh}{\mathbf{h}}

\newcommand{\bt}{\mathbf{t}}
\newcommand{\bu}{\mathbf{u}}

\newcommand{\bw}{\mathbf{w}}
\newcommand{\bx}{\mathbf{x}}
\newcommand{\by}{\mathbf{y}}
\newcommand{\bz}{\mathbf{z}}

\newcommand{\bC}{\mathbf{C}}
\newcommand{\ba}{\mathbf{a}}
\newcommand{\bb}{\mathbf{b}}
\newcommand{\bc}{\mathbf{c}}

\endlocaldefs

\begin{document}

\begin{frontmatter}
\title{Uniform-over-dimension convergence with application to location tests for high-dimensional data}
%\title{A sample article title with some additional note\thanksref{T1}}
\runtitle{Uniform-over-dimension convergence}
%\thankstext{T1}{A sample of additional note to the title.}

\begin{aug}
%%%%%%%%%%%%%%%%%%%%%%%%%%%%%%%%%%%%%%%%%%%%%%%
%% Only one address is permitted per author. %%
%% Only division, organization and e-mail is %%
%% included in the address.                  %%
%% Additional information can be included in %%
\author[A]{\fnms{Joydeep}~\snm{Chowdhury}\ead[label=e1]{joydeep.chowdhury@kaust.edu.sa}},
\author[B]{\fnms{Subhajit}~\snm{Dutta}\ead[label=e2]{duttas@iitk.ac.in}}
\and
\author[A]{\fnms{Marc G.}~\snm{Genton}\ead[label=e3]{marc.genton@kaust.edu.sa}}
%%%%%%%%%%%%%%%%%%%%%%%%%%%%%%%%%%%%%%%%%%%%%%
%% Addresses                                %%
%%%%%%%%%%%%%%%%%%%%%%%%%%%%%%%%%%%%%%%%%%%%%%
\address[A]{King Abdullah University of Science and Technology\printead[presep={,\ }]{e1,e3}}

\address[B]{Indian Institute of Technology Kanpur\printead[presep={,\ }]{e2}}
\end{aug}

\begin{abstract}
Asymptotic methods for hypothesis testing in high-dimensional data usually require the dimension of the observations to increase to infinity, often with an additional condition on its rate of increase compared to the sample size. On the other hand, multivariate asymptotic methods are valid for fixed dimension only, and their practical implementations in hypothesis testing methodology typically require the sample size to be large compared to the dimension for yielding desirable results. However, in practical scenarios, it is usually not possible to determine whether the dimension of the data at hand conform to the conditions required for the validity of the high-dimensional asymptotic methods, or whether the sample size is large enough compared to the dimension of the data. In this work, a theory of asymptotic convergence is proposed, which holds uniformly over the dimension of the random vectors. This theory attempts to unify the asymptotic results for fixed-dimensional multivariate data and high-dimensional data, and accounts for the effect of the dimension of the data on the performance of the hypothesis testing procedures. The methodology developed based on this asymptotic theory can be applied to data of any dimension. An application of this theory is demonstrated in the two-sample test for the equality of locations. The test statistic proposed is unscaled by the sample covariance, similar to usual tests for high-dimensional data. Using simulated examples, it is demonstrated that the proposed test exhibits better performance compared to several popular tests in the literature for high-dimensional data. Further, it is demonstrated in simulated models that the proposed unscaled test performs better than the usual scaled two-sample tests for multivariate data, including the Hotelling's $T^2$ test for multivariate Gaussian data.
\end{abstract}



\end{frontmatter}
%%%%%%%%%%%%%%%%%%%%%%%%%%%%%%%%%%%%%%%%%%%%%%
%% Please use \tableofcontents for articles %%
%% with 50 pages and more                   %%
%%%%%%%%%%%%%%%%%%%%%%%%%%%%%%%%%%%%%%%%%%%%%%
%\tableofcontents

%%%%%%%%%%%%%%%%%%%%%%%%%%%%%%%%%%%%%%%%%%%%%%
%%%% Main text entry area:
\section{Introduction}
\label{introduction}
In diverse studies and scientific experiments, observations generated are high-dimensional in nature, i.e., the dimension of the observation vectors are higher than the number of observations (see, e.g., \cite{carvalho2008high,wang2008approaches,lange2014next,buhlmann2014high}).
There is an extensive literature on high-dimensional one-sample and multi-sample tests of means, or more generally, of the locations/centers of the underlying distributions (\cite{goeman2006testing,chen2010two,srivastava2013two,biswas2014nonparametric,cai2014two,wang2015high,javanmard2014confidence,chakraborty2017tests}).
In many asymptotic results involving high-dimensional data, authors make the assumption that the dimension $p \to \infty$ along with sample size $n \to \infty$ (see, e.g., \cite{bai1996effect,chen2010two,zhang2020simple}). In many other results, there are conditions on the rate of growth of $p$, e.g., $p = o( n^\alpha )$ for some $\alpha$ (see, e.g., \cite{hu2020pairwise}). However, these conditions on the dimension are not easy to verify, and it is often unclear whether they are satisfied in a particular situation or not. In a practical situation of testing, an experimenter has only one sample of observations, which fixes the value of $n$ and $p$. Based on just this one pair of $n$ and $p$, it is not possible to verify whether $p$ satisfies the particular rate with respect to $n$ required for the validity of those aforementioned results.

Multivariate asymptotic results for an extensive collection of statistics are well-established in the literature. These results are based on the assumption of a fixed value of $p$ and taking $n$ to infinity. However, in practice, except when $p$ is very small compared to $n$, sometimes it is observed that existing asymptotic results provide unsatisfactory approximations to the actual distributions of the test statistics for moderate values of $p$, which are still far lower compared to $n$.

In this work, we develop an asymptotic theory of convergence which holds uniformly over the dimension $p$.	
This eliminates the concern over the validity of the asymptotic conditions on the dimension, which are generally found in existing results.
The asymptotic results developed using this theory can be applied to data with arbitrary dimensions, including usual multivariate as well as high-dimensional data.

Suppose one carries out a test of hypothesis based on $n$ independent $p$-dimensional observations $\bY_{1, p}, \ldots, \bY_{n, p}$. The test statistic in any such hypothesis test, $T_{n, p} = g_p( \bY_{1, p}, \ldots, \bY_{n, p} )$ is always univariate, whatever the value of the dimension $p$ of the data might be.
This observation motivates the setup we consider for developing the uniform-over-dimension asymptotic theory of convergence. We define uniform-over-dimension convergence in distribution of functions of random vectors and state its associated results, which are analogous to the usual results for convergence in distribution. More specifically, we consider functions $\bX_p = f_p( \bY_{n, p} )$ of $p$-dimensional random vectors $\bY_{n, p}$, where $f_p : \mathbb{R}^p \to \mathbb{R}^d$ and $d$ is some positive integer, and define the uniform-over-$p$ convergence of $\bX_p$. For a test statistic, $d$ is 1, but we consider general integer values of $d$ while developing the theory, which helps in deriving subsequent results. Often, a test statistic can be decomposed in two components, one of which is asymptotically negligible, and it is easier to derive the asymptotic distribution of the other component. To implement similar techniques in deriving uniform-over-dimension asymptotic distributions, we define uniform-over-dimension convergence in probability of functions of random vectors, and state associated results concerning both uniform-over-dimension convergence in distribution and convergence in probability.

The proposed theory is demonstrated on a test for equality of locations and deriving the uniform-over-dimension asymptotic null distribution of the test statistic. The proposed test statistic is constructed based on a kernel, without using normalization by the sample covariance matrix. Normalization by the sample covariance is not usually used for high-dimensional tests, and we have followed a similar principle here. However, we propose the test to be applied without regard to the dimension of the data, whether high-dimensional or standard multivariate situations. A natural question arises about the performance of the proposed testing method without normalization in case of multivariate data, where normalization by sample covariance is usually employed. However, we have shown using simulated as well as real data analyses that the proposed test outperforms the other tests, both in high-dimensional and lower-dimensional multivariate data, in Gaussian as well as non-Gaussian models.

In \autoref{sec:theory}, the definitions and theorems related to uniform-over-dimension convergence in distribution and probability of functions of random vectors are presented. In \autoref{sec:multisample}, the proposed theory is employed to derive the asymptotic null distribution of a suggested test of equality of locations of several populations, which is valid uniformly over the dimension of the observations. Further, the asymptotic consistency of the test is also established uniformly over the dimension of the observations. In \autoref{sec:data_analysis}, it is demonstrated using simulated and real data that the proposed test equipped with its asymptotic null distribution valid uniformly over the data dimension outperforms other tests available for both high-dimensional and usual multivariate data. In \autoref{sec:conclusion}, further work on the proposed theory and potential applications are discussed. Proofs of the mathematical results are presented in the Appendix.

\section{Main theoretical results}\label{sec:theory}
In this section, we define uniform-over-dimension convergence in distribution and in probability, and state the associated results. In the first subsection, the uniform-over-dimension convergence in distribution is defined and its associated results are stated. The uniform-over-dimension convergence in probability is defined in the second subsection, and results involving both the notions of convergence are stated there.

\subsection{Definitions and Theorems}
\label{subsec:unifdimdist}
Let $\{ \bX_{n, p} \}$ be $d$-dimensional random vectors with probability measures $\mu_{n, p}$ and associated distribution functions $F_{n, p}( \cdot )$ on $(\mathbb{R}^d, \mathcal{R}^d)$, where $\mathcal{R}^d$ is the Borel sigma field on $\mathbb{R}^d$, and the indices $n$ and $p$ are positive integers. Also, let $\bX_p$ be a $d$-dimensional random vector with probability measure $\mu_p$ and associated distribution function $F_p( \cdot )$ on $(\mathbb{R}^d, \mathcal{R}^d)$.
We think of $\bX_{n, p}$ as a fixed-dimensional vector-valued function of some $p$-dimensional random vector $\bY_{n, p}$, and similarly $\bX_p$ is thought of as a fixed-dimensional vector-valued function of a $p$-dimensional random vector $\bY_p$.
Although the motivation behind the following theory is the fact that a test statistic is always univariate (i.e., $d=1$) irrespective of the dimension of the underlying observations, we are considering vector-valued functions of the observations instead of univariate functions. This consideration turns out to be useful in the proofs of some of the results on uniform-over-dimension convergence of random variables.

\begin{definition}\label{definition1}
We say that $\bX_{n, p}$ converges in distribution to $\bX_p$ uniformly-over-$p$ if
for every bounded continuous function $f : \mathbb{R}^d \to \mathbb{R}$,
\begin{align*}
\lim\limits_{n \to \infty} \sup_p \left| \int f \mathrm{d} \mu_{n, p} - \int f \mathrm{d} \mu_p \right| = 0 .
\end{align*}
We write it as either of the following:
$\bX_{n, p} \Longrightarrow \bX_p$ uniformly-over-$p$, or $\mu_{n, p} \Longrightarrow \mu_p$ uniformly-over-$p$, or $F_{n, p} \Longrightarrow F_p$ uniformly-over-$p$.
\end{definition}

We proceed to state and prove several results for uniform-over-$p$ convergence in distribution, which are analogous to theorems for the usual weak convergence of probability measures (see, e.g., \cite{billingsley2013convergence}). The following assumption is required for those subsequent results:
\begin{enumerate}[label = (A\arabic*), ref = (A\arabic*)]
\item The collection of probability measures $\{ \mu_p \}$ on $\mathbb{R}^d$ is relatively compact with respect to the total variation metric.
\label{assumption1}
\end{enumerate}
Assumption \ref{assumption1} means that every sequence $\{ \mu_{p_n} \}$ in $\{ \mu_p \}$ has a subsequence $\{ \mu_{p_{n_k}} \}$ such that $\sup_{C \in \mathcal{R}^d} | \mu_{p_{n_k}}( C ) - \nu( C ) | \to 0$ as $k \to \infty$ for some probability measure $\nu$ on $\mathbb{R}^d$, where $\nu$ may not be a member of the collection $\{ \mu_p \}$. Now, $\sup_{C \in \mathcal{R}^d} | \mu_{p_{n_k}}( C ) - \nu( C ) | \to 0$ as $k \to \infty$ implies that $\mu_{p_{n_k}}( C ) \to \nu( C )$ as $k \to \infty$ for any $\nu$-continuity set $C$, i.e., $\{ \mu_{p_{n_k}} \}$ weakly converges to $\nu$. Since $\mathbb{R}^d$ is separable and complete, it follows that $\{ \mu_p \}$ is tight (see pp.~57, 59 and Theorem 5.2 in \cite{billingsley2013convergence}), i.e.,
given any $\epsilon > 0$, there is a compact set $K \subset \mathbb{R}^d$ such that
\begin{align}
\mu_p(K) > 1 - \epsilon \text{ for every } \mu_p .
\label{tightness}
\end{align}
Further, from the separability of $\mathbb{R}^d$, it follows that $\{ \mu_p \}$ is relatively compact with respect to the metric of weak convergence of probability measures (see point (iv) in page 72 in \cite{billingsley2013convergence}).

Assumption \ref{assumption1} implies tightness, which is a common assumption while establishing convergence in distribution for a diverse collection of measures (see, e.g., Theorems 7.1, 7.3, 13.1 and 13.2 in \cite{billingsley2013convergence}). Although assumption \ref{assumption1} is stronger than the common assumption of tightness, it is satisfied in diverse cases including those of common interest.

We need the following lemma to prove the uniform-over-$p$ analogue of the Portmanteau theorem (cf. Lemma 8.5.1 in \cite{lebanon2012probability}).

\begin{lemma}\label{lemma1}
Under assumption \ref{assumption1}, given any Borel set $A \in \mathcal{R}^d$, which is a $\mu_p$-continuity set for all $p$, and any $\epsilon > 0$, there are bounded and Lipschitz continuous functions $g, h : \mathbb{R}^d \to \mathbb{R}$ such that $g \le I_A \le h$ and $\sup_p \int (h - g) \mathrm{d} \mu_p < \epsilon$, where $I_A( \cdot )$ is the indicator function of the set $A$.
\end{lemma}

We state the uniform-over-$p$ analogue of the Portmanteau theorem (Theorem 2.1 in \cite{billingsley2013convergence}) below.

\begin{theorem}[Uniform-over-$p$ Portmanteau theorem]\label{portmanteau}
Let $\mu_{n, p}, \mu_p, F_{n, p}( \cdot )$ and $F_p( \cdot )$ be as defined above, and assumption \ref{assumption1} be satisfied.
Then, the following are equivalent:
\begin{enumerate}[label = (\roman*), ref = (\roman*)]
\item \label{p1} $ \mu_{n, p} \Longrightarrow \mu_p $ uniformly-over-$p$.

\item \label{p2} For every bounded and uniformly continuous function $ f : \mathbb{R}^d \to \mathbb{R} $,
\begin{align*}
\lim\limits_{n \to \infty} \sup_p \left| \int f \mathrm{d} \mu_{n, p} - \int f \mathrm{d} \mu_p \right| = 0 .
\end{align*}
	
\item \label{p3} For every continuous function $ f : \mathbb{R}^d \to \mathbb{R} $, which is zero outside of a compact set,
\begin{align*}
\lim\limits_{n \to \infty} \sup_p \left| \int f \mathrm{d} \mu_{n, p} - \int f \mathrm{d} \mu_p \right| = 0 .
\end{align*}

\item \label{p4} For every bounded and Lipschitz continuous function $ f : \mathbb{R}^d \to \mathbb{R} $,
\begin{align*}
\lim\limits_{n \to \infty} \sup_p \left| \int f \mathrm{d} \mu_{n, p} - \int f \mathrm{d} \mu_p \right| = 0 .
\end{align*}
	
\item \label{p5} For every Borel set $ A $, which is a $ \mu_p $-continuity set for all $p$,
\begin{align*}
\lim\limits_{n \to \infty} \sup_p | \mu_{n, p}( A ) - \mu_p( A ) | = 0 .
\end{align*}
	
\item \label{p6} For $ \bx $ being a continuity point of $ F_p( \cdot ) $ for all $p$,
\begin{align*}
\lim\limits_{n \to \infty} \sup_p | F_{n, p}( \bx ) - F_p( \bx ) | = 0 .
\end{align*}
\end{enumerate}
\end{theorem}

Based on the uniform-over-$p$ Portmanteau theorem, we proceed to state the uniform-over-$p$ L\'{e}vy's continuity theorem.

\begin{theorem}[Uniform-over-$p$ L\'{e}vy's continuity theorem]\label{levy}
Let $ \mu_{n, p} $ and $ \mu_p $ be probability measures on $\mathbb{R}^d$ with characteristic functions $ \varphi_{n, p} $ and $ \varphi_p $. Then, $ \mu_{n, p} \Longrightarrow \mu_p $ uniformly-over-$p$ if and only if for every $ \bt \in \mathbb{R}^d $, we have
\begin{align*}
\lim\limits_{n \to \infty} \sup_p | \varphi_{n, p}( \bt ) - \varphi_p( \bt ) | = 0 .
\end{align*}
\end{theorem}

The uniform-over-$p$ L\'{e}vy's continuity theorem is the most critical tool to establish the uniform-over-$p$ convergence in distribution of a sequence of random variables which are functions of some $p$-variate random vectors. Using this theorem, one needs to only establish that the characteristic functions of those random variables converge uniformly over $p$ to establish their uniform-over-$p$ convergence in distribution. 

The uniform-over-$p$ continuous mapping theorem stated below is also a useful result.
\begin{theorem}[Uniform-over-$p$ continuous mapping theorem]\label{mappingthm}
Let $\bX_{n, p}$ and $\bX_p$ be $d$-dimensional random vectors with $\bX_{n, p} \Longrightarrow \bX_p$ uniformly-over-$p$ and $g : \mathbb{R}^d \to \mathbb{R}$ be a continuous function. Then, $g\left( \bX_{n, p} \right) \Longrightarrow g\left( \bX_p \right)$ uniformly-over-$p$.
\end{theorem}

\subsection{Uniform-over-$p$ Convergence in Probability}
In the last section, we studied the asymptotic convergence in distribution of fixed-dimensional random vectors which are functions of possibly high-dimensional random vectors. In this subsection, we concentrate on uniform-over-dimension convergence in probability of fixed-dimensional random vectors which are functions of possibly high-dimensional random vectors.

\begin{definition}\label{definition2}
Let $\{ \bX_{n, p} \}$ and $\{ \bX_{p} \}$ be $d$-dimensional random vectors. We say that $\bX_{n, p} \stackrel{\P}{\longrightarrow} \bX_{p}$ uniformly-over-$p$ if for every $ \epsilon > 0 $,
\begin{align*}
\lim\limits_{n \to \infty} \sup_p \P\left[ \left\| \bX_{n, p} - \bX_{p} \right\| > \epsilon \right] = 0 .
\end{align*}
\end{definition}

While applying the tools to establish uniform-over-dimension convergence in distribution for a sequence of real-valued random variables, one may come across a corresponding sequence of real-valued random variables such that it is easier to establish the uniform-over-dimension convergence in distribution of the second sequence of random variables and the difference between the two sequences converges to zero in probability uniformly-over-dimension. In such cases, the uniform-over-dimension convergence in distribution for the first sequence can be established using the uniform-over-dimension analogue of Slutsky's theorem.
The following lemma is required to prove the uniform-over-$p$ Slutsky's theorem.

\begin{lemma}\label{lemmaSlutsky}
Let $\{ X_{n, p} \}$, $\{ Y_{n, p} \}$ and $\{ X_p \}$ be real-valued random variables and $\{ c_p \}$ be a bounded sequence of real numbers with $X_{n, p} \Longrightarrow X_p$ uniformly-over-$p$ and $Y_{n, p} \stackrel{\P}{\longrightarrow} c_p$ uniformly-over-$p$. Then, $\left( X_{n, p}, Y_{n, p} \right)^\top \Longrightarrow \left( X_p, c_p \right)^\top$ uniformly-over-$p$.
\end{lemma}

Now, the uniform-over-$p$ Slutsky's theorem follows from \autoref{mappingthm} and \autoref{lemmaSlutsky}.

\begin{theorem}[Uniform-over-$p$ Slutsky's theorem]\label{Slutsky}
Let $\{ X_{n, p} \}$, $\{ Y_{n, p} \}$ and $\{ X_p \}$ be real-valued random variables and $\{ c_p \}$ be a bounded sequence of real numbers such that $X_{n, p} \Longrightarrow X_p$ uniformly-over-$p$ and $Y_{n, p} \stackrel{\P}{\longrightarrow} c_p$ uniformly-over-$p$. Then,
\begin{enumerate}[label = (\alph*), ref = (\alph*)]
\item $X_{n, p} + Y_{n, p} \Longrightarrow X_p + c_p$ uniformly-over-$p$,
\label{s1}

\item $ X_{n, p} Y_{n, p} \Longrightarrow c_p X_p$ uniformly-over-$p$,
\label{s2}

\item $ X_{n, p} / Y_{n, p} \Longrightarrow X_p / c_p$ uniformly-over-$p$ if $\{ c_p \}$ is bounded away from 0, i.e., $\inf_p | c_p | > 0$.
\label{s3}
\end{enumerate}
\end{theorem}

\section{Multi-sample location tests}\label{sec:multisample}
In this section, we demonstrate how the preceding theory of convergence helps us derive asymptotic distributions of test statistics independent of any assumption on the dimension.

Let $\nu_{p, 0}$ be a probability measure on $\mathbb{R}^p$, $\ba_{p, 1}, \ldots, \ba_{p, K}$ be $p$-dimensional vectors and $\nu_{p, k}$ denote the linear shift of $\nu_{p, 0}$ by $\ba_{p, k}$, i.e., if the random vector $\bY_p$ has the probability measure $\nu_{p, 0}$, then $\nu_{p, k}$ is the probability measure corresponding to $\bY_p + \ba_{p, k}$. We consider a sample of size $n$ consisting of $n_k$ independent observations generated from the probability measures $\nu_{p, k}$  for $k = 1, \ldots, K$. The observations from the measure $\nu_{p, k}$ are denoted as $\bY_{p, k, i}$, where $i = 1, \ldots, n_k$ and $k = 1, \ldots, K$.
We are interested in testing the hypothesis
\begin{align} \label{h_0}
\text{H}_0 : \; \ba_{p, 1} = \ba_{p, 2} = \cdots = \ba_{p, K} \;.
\end{align}

Let $\bh_p( \cdot, \cdot ) : \mathbb{R}^p \times \mathbb{R}^p \to \mathbb{R}^p$ satisfy that $\bh_p( \bx, \by ) = - \bh_p( \by, \bx )$ for all $\bx, \by \in \mathbb{R}^p$. Define
$\bR_p( \by ) = n^{-1} \sum_{k=1}^{K} \sum_{i=1}^{n_k} \bh_p( \by, \bY_{p, k, i} )$.
Also, define $\bar{\bR}_{p, k} = n_k^{-1} \sum_{i = 1}^{n_k} \bR_p\left( \bY_{p, k, i} \right)$.
Note that $\E[ \bar{\bR}_{p, k} ] = n^{-1} \sum_{l=1}^{K} n_l \allowbreak \E[ \bh_p( \bY_{p, k, 1}, \bY_{p, l, 1} ) ] = \mathbf{0}_p$ for all $k$ under $\text{H}_0$ in \eqref{h_0}. So, a test of $\text{H}_0$ may be conducted based on the magnitudes of $ \bar{\bR}_{p, k} $ for $ k = 1, \ldots, K $, and a high value of $\left\| \bar{\bR}_{p, k} \right\|$ for any $k$ would indicate that $ \text{H}_0 $ in \eqref{h_0} is not true. Define the test statistic
\begin{align*}
S_{n, p} = \sum_{k=1}^{K} n_k \left\| \bar{\bR}_{p, k} \right\|^2 .
\end{align*}
The hypothesis $ \text{H}_0 $ in \eqref{h_0} is to be rejected if the value of $ S_{n, p} $ is large.

When $\bh_p( \bx, \by ) = \bx - \by$ and $K = 2$, we get the test proposed by \cite{zhang2020simple}. We shall denote this test by \cite{zhang2020simple} as ZGZC2020 test in our subsequent data analysis. If $\bh_p( \bx, \by ) = ( \bx - \by ) / \| \bx - \by \|$ but the observations are from a separable Hilbert space $\mathcal{H}$, then we get the test proposed by \cite{chowdhury2022multi}. We shall denote this test as the SS test. A similar test was proposed by \cite{choi1997approach}, but they normalized the test statistic using the estimated covariance of the vector of $\bar{\bR}_{p, k}$. Another test based on spatial signs for the two-sample test was proposed in \cite{feng2016multivariate}, but the authors there normalize the test statistic using a diagonal matrix. Further, the asymptotic theory for that test was derived assuming certain conditions on the growth of $p$ with $n$.

To derive the asymptotic distribution of the test statistic, we need the following lemma.

\begin{lemma}\label{lemma_gaussian}
Let $\bZ_{p, 1}, \ldots, \bZ_{p, n}$ be independent $p$-dimensional random vectors with $\E[ \bZ_{p, 1} ] = \mathbf{0}_p$ and $\Var( \bZ_{p, 1} ) = \bSigma_p$. Let $\bar{\bZ}_{p} = n^{-1} \sum_{i=1}^n \bZ_{p, i}$ and $\delta_{p, 1} \ge \cdots \ge \delta_{p, p}$ be the $p$-eigenvalues of $\bSigma_p$.
Suppose the following are satisfied:
\begin{enumerate}[label = (\alph*), ref = (\alph*)]
\item $n_k / n \to \gamma_k \in (0, 1)$ for $k = 1, \ldots, K$.

\item There exists $\rho_k$ for $k = 1, 2, \ldots$, such that $\lim_{p \to \infty} \left( \delta_{p, k} / \sqrt{\sum_{k = 1}^p \delta_{p, k}^2} \right) = \rho_k$ uniformly over $k$ and $\sum_{k = 1}^\infty \rho_k < \infty$.
\end{enumerate}
Then,
\begin{align*}
\frac{n \left\| \bar{\bZ}_{p} \right\|^2 - \sum_{k=1}^p \delta_{p, k}}{\left( \sum_{k=1}^p \delta_{p,k}^2 \right)^{1/2}} \Longrightarrow \frac{\sum_{k=1}^p \delta_{p, k} W_k}{\left( \sum_{k=1}^p \delta_{p,k}^2 \right)^{1/2}}
\end{align*}
uniformly-over-p as $n \to \infty$.
\end{lemma}

The following theorem yields the asymptotic uniform-over-$p$ null distribution of our proposed testing procedure.

\begin{theorem}\label{thm:1}
Assume that $ n^{-1} n_k \to \lambda_k \in ( 0, 1 ) $ for all $ k $ as $ n \to \infty $.
Let $ \bX_{p, i} $, $ \bY_{p, j} $ and $ \bZ_{p, k} $ be independent random vectors having distributions of the $ i^\text{th} $, the $ j^\text{th} $ and the $ k^\text{th} $ groups of the sample, respectively.
Define
\begin{align*}
\bC_p( i, j, k ) = \Cov\left( \E\left[ \bh_p\left( \bX_{p, i} , \bZ_{p, k} \right) \mid \bZ_{p, k} \right] ,\, \E\left[ \bh_p\left( \bY_{p, j} , \bZ_k \right) \mid \bZ_k \right] \right)
\end{align*}
and 
$ \bSigma_p = \left( \boldsymbol{\sigma}_{p, k_1, k_2} \right)_{K \times K} $, where
\begin{align*}
\boldsymbol{\sigma}_{p, k_1, k_2}
& = \sqrt{\lambda_{p, k_1} \lambda_{p, k_2}} \sum_{l=1}^{K} \lambda_{p, l}
\left[ \bC_p( k_1, k_2, l ) - \bC_p( l, k_2, k_1 ) - \bC_p( k_1, l, k_2 ) \right] \\
& \quad + \sum_{l_1 = 1}^{K} \sum_{l_2 = 1}^{K} \lambda_{l_1} \lambda_{l_2} \bC_p( l_1, l_2, k_1 ) \mathbb{I}( k_1 = k_2 ) .
\end{align*}
Let $\{ \gamma_{p, i} \}$ be the $p K$ eigenvalues of the matrix $\bSigma_p$, which satisfy condition (b) in \autoref{lemma_gaussian}.
Then,
\begin{align*}
\frac{\sum_{k=1}^{K} n_k \left\| \bar{\bR}_{p, k} - \E[ \bar{\bR}_{p, k} ] \right\|^2}{\left( 2 \sum_{i = 1}^{p K} \gamma_{p, i}^2 \right)^{1/2}} \Longrightarrow \frac{\sum_{i = 1}^{p K} \gamma_{p, i} U_i}{\left( 2 \sum_{i = 1}^{p K} \gamma_{p, i}^2 \right)^{1/2}}
\end{align*}
uniformly-over-$p$ as $ n \to \infty $, where $\{ U_i \}$ are $p K$ independent random variables following the $\chi^2$ distribution with degree of freedom 1.
\end{theorem}

From \autoref{thm:1}, we get the uniformly-over-$p$ asymptotic null distribution of the proposed kernel-based test. It coincides with the asymptotic null distribution for $n, p \to \infty$ of the test by \cite{zhang2020simple} when $\bh_p( \bx, \by ) = \bx - \by$ and $K = 2$. The implementation of the test is also the same as in \cite{zhang2020simple}.
However, the choice $\bh_p( \bx, \by ) = \bx - \by$ requires the existence of mean and variance of the random vectors, which may not be the case for some heavy-tailed distributions, e.g., the multivariate Cauchy distribution. On the other hand, the choice $\bh_p( \bx, \by ) = ( \bx - \by ) / \| \bx - \by \|$ ensures all the moments exist, because the kernel $\bh_p( \bx, \by )$ itself is bounded. This choice can be more widely applied without regard to the existence of moments of the underlying random vectors.

Based on \autoref{lemma_gaussian} and \autoref{thm:1}, we establish the following result, which implies the uniform-over-$p$ asymptotic unbiasedness and the uniform-over-$p$ asymptotic consistency of our proposed testing method.

\begin{theorem}\label{thm:2}
Let $ n^{-1} n_k \to \lambda_k \in ( 0, 1 ) $ for all $ k $ as $ n \to \infty $, and let $\{ \gamma_{p, i} \}$ be the $p K$ eigenvalues of the matrix $\bSigma_p$ defined in \autoref{thm:1}, which satisfy condition (b) in \autoref{lemma_gaussian}.
\begin{enumerate}[label=(\alph*)]
\item Under $ \text{H}_0 $ in \eqref{h_0} and for every $ 0 < \alpha < 1 $, the size of the level $ \alpha $ kernel-based test converges to $ \alpha $ as $ n \to \infty $ uniformly-over-$p$.

\item Next, suppose $ \nu_{p,0} $ and $\mathbf{h}_p( \cdot, \cdot )$ satisfy the following:
\begin{itemize}
\item If $\bX_{p}$ and $\bY_{p}$ are independent random vectors with distribution $\nu_{p,0}$ and $\bb_{p,1} \neq \bb_{p,2}$, then $E[ \mathbf{h}_p( \bb_{p,1} + \bX_{p}, \bb_{p,2} + \bY_{p} ) ] \neq \mathbf{0}_p$ for all $p$.
\end{itemize}
In this case, if $ \text{H}_0 $ in \eqref{h_0} is not true,
we have $ \sum_{l=1}^{K} \lambda_l E[ \mathbf{h}_p( \mathbf{X}_{k 1} - \mathbf{X}_{l 1} ) ] \neq \mathbf{0}_p $ for every $ k $, and hence for every $ 0 < \alpha < 1 $, the power of the level $ \alpha $ kernel-based test converges to $ 1 $ as $ n \to \infty $ uniformly-over-$p$.
\end{enumerate}
\end{theorem}

The condition stated in the above theorem for the uniform-over-$p$ asymptotic consistency is satisfied when $\mathbf{h}_p( \bx_p, \by_p ) = \bx_p - \by_p$ and the expectation of $\nu_{p, 0}$ is $\mathbf{0}_p$ for all $p$. When $\mathbf{h}_p( \bx_p, \by_p ) = (\bx_p - \by_p) / \| \bx_p - \by_p \|$, the condition is satisfied if the support of $ \nu_{p,0} $ is not contained in a lower-dimensional subspace (including an affine subspace) of $\mathbb{R}^p$ (see the proof of Theorem 2.6 in \cite{chowdhury2022multi}).

We implement the kernel-based test for the two aforementioned kernels. For the first choice of the kernel, for which the test coincides with that proposed in \cite{zhang2020simple}, we follow the implementation procedure described in \cite{zhang2020simple}. For the spatial sign kernel, we follow the implementation procedure described in section 2 of \cite{chowdhury2022multi}.

In the next section, we compare the performances of the two tests, which we obtain as particular cases of the general kernel-based test, with other popular tests in high-dimensional and usual multivariate data.

\section{Comparison of performance}\label{sec:data_analysis}
In this section, we demonstrate and compare the performances of the test by \cite{zhang2020simple} and our proposed SS test with several tests available in the literature using simulated and real data. This section is divided into three subsections. In the first subsection, we compare the performances of the test by \cite{zhang2020simple} and our proposed SS test with several tests in the literature for testing equality of locations in two samples of high-dimensional data, when the dimension of the sample is either similar or larger than the sample size. In the second subsection, we investigate and compare the performances of the aforementioned two tests with the Hotelling's $T^2$ test and the test by \cite{choi1997approach} for equality of locations in two samples when the dimension of the data is very small compared to the sample size. In the third subsection, we demonstrate and compare the performances of the tests in a real dataset.

Although our proposed SS test is applicable for testing equality of locations when the number of groups is larger than two, we have focused here on the comparison of performances of two-sample tests. In \autoref{subsec:simulation_highdim}, we consider several simulation models where the dimension $p$ is either close to or larger than the sample sizes. There, we estimate the powers of the test by \cite{zhang2020simple} and the SS test along with several other popular two-sample tests in the literature for high-dimensional data. We consider both Gaussian and non-Gaussian distributions and investigate the sizes and powers of the tests in those setups.

Recall that the test by \cite{zhang2020simple} and the SS test do not normalize the test statistic with any estimated sample covariance matrix. In fact, the test statistic by \cite{zhang2020simple} is similar to the popular Hotelling's $T^2$ two-sample test statistic except for the normalization by the sample covariance matrix. Similarly, the SS test does not employ normalization by a sample covariance matrix unlike the test proposed in \cite{choi1997approach} as a generalization of the Kruskal-Wallis test for multivariate data. Testing procedures proposed for multivariate data usually employ normalization by the sample covariance, because it makes the test statistics scale invariant. On the other hand, a test, which is to be applied for high-dimensional data, cannot employ similar normalization because a sample covariance matrix would not be invertible when the dimension is larger than the sample size. However, it is of interest to investigate whether there is any advantages in terms of statistical power of multivariate tests which are such normalized against the two un-normalized procedures we have proposed here. Thus, in \autoref{subsec:simulation_scaling}, we compare the performances of the test by \cite{zhang2020simple}, our proposed SS test with the Hotelling's $T^2$ test and the test by \cite{choi1997approach} in several simulation models, where the dimension of the sample is very small compared to the sample sizes, and the underlying distributions are Gaussian or non-Gaussian.

In the previous section, we have studied the test by \cite{zhang2020simple} and our proposed SS test as procedures to apply for any two-sample testing situation irrespective of the dimension of the data. So, it is imperative to investigate their performances in both high-dimensional and lower dimensional multivariate settings, which we do in \autoref{subsec:simulation_highdim} and \autoref{subsec:simulation_scaling}. As mentioned before, we call the test by \cite{zhang2020simple} as ZGZC2020 test.

\subsection{Comparison with high-dimensional tests}\label{subsec:simulation_highdim}
Here, we compare the performance of several two-sample tests available in the literature for high-dimensional data with the ZGZC2020 test and the SS test via several simulation models. The sample sizes for the two samples are taken as $n_1 = 40$ and $n_2 = 50$, while the values of the dimension $p$ considered are 30, 50 and 100. When $p$ is 30 or 50, the sample sizes $n_1$ and $n_2$ are similar to the dimension $p$; larger than the dimension for $p = 30$ and smaller or equal to the dimension when $p = 50$. However, for $p = 100$, $n_1$ and $n_2$ are substantially smaller than the dimension. The lowest value of $p$ considered here is not too small compared to the sample sizes, the sample covariance matrix would be invertible, but high-dimensional methods can be applied as well. For the two larger values of $p$, the sample covariance matrix cannot be inverted but high-dimensional methods can be applied justifiably. For the underlying distributions of the samples, we consider three simulation models as described below. One of them is Gaussian, one is non-Gaussian but with finite variance and the third one is the multivariate Cauchy distribution, a heavy tailed non-Gaussian distribution with non-existent expectation.

Let $\bSigma_p = (( \sigma_{i j} ))$ with $\sigma_{i j} = 0.5 + 0.5 I(i = j)$. Recall the descriptions of the probability measures $\nu_{p, k}$ and $\nu_{p, 0}$ from \autoref{sec:multisample}. We consider three choices of $\nu_{p, 0}$:
\begin{itemize}
    \item \textit{Model 1}: $\nu_{p, 0}$ is $G_p( \mathbf{0}_p, \bSigma_p )$, the $p$-dimensional zero-mean Gaussian distribution with covariance $\bSigma_p$,

    \item \textit{Model 2}: $\nu_{p, 0}$ is $t_{p, 4}( \mathbf{0}_p, \bSigma_p )$, the $p$-dimensional $t$ distribution with degree of freedom 4, location vector $\mathbf{0}_p$ and scale matrix $\bSigma_p$,

    \item \textit{Model 3}: $\nu_{p, 0}$ is $t_{p, 1}( \mathbf{0}_p, \bSigma_p )$, the $p$-dimensional Cauchy distribution with location vector $\mathbf{0}_p$ and scale matrix $\bSigma_p$.
\end{itemize}
Let $\bu_p = (1, 2, \ldots, p)^\top$ and $\bh_p = \bu_p / \| \bu_p \|$.
For each of the three models, we choose the shifts $\ba_{p, 1} = \mathbf{0}_p$ and $\ba_{p, 2} = \delta \bh_p$, where $\delta$ is a real number.

Both Model 2 and Model 3 are comprised of non-Gaussian distributions with heavy tails. Model 2 has up to third order moments while no moment exists for the distribution in Model~3.

When $\delta = 0$, the null hypothesis \eqref{h_0} is satisfied, whereas the disparity between the two populations increases as $\delta$ is increased. So, we vary $\delta$ on a grid of values from 0 to some positive numbers and estimate the powers of the tests for each value of $\delta$, and thus generate estimated power curves of the tests over the corresponding ranges of $\delta$. We consider different values of $p$ and fix the sizes of the two groups to focus on how the powers of the tests change with $p$.

The distributions in Models 1 and 2 are either identical or very similar to the distributions considered in Models 1 and 2 in subsection 3.1 in \cite{zhang2020simple}. The expression of the shifts are also the same, however, the values of $\delta$ considered and the values of the dimension $p$ and the sample sizes $n_1$ and $n_2$ are different. We choose these models because of their simplicity and convenience of comparison of the performances of the ZGZC2020 test by \cite{zhang2020simple} and our proposed SS test, along with several other tests from the literature. We consider Model 3 in addition to investigate the performances of the ZGZC2020 test and the other tests from the literature under significant departure from Gaussianity, so much so that the moments required for the theoretical validity of those tests are non-existent, while the proposed SS test does not require the existence of population moments.

From the literature of two-sample tests for high-dimensional data, we consider the tests by \cite{bai1996effect}, \cite{cai2014two}, \cite{chen2010two}, \cite{chen2014two} and \cite{srivastava2008test}, and denote them as BS1996 (\cite{bai1996effect}), CLX2014 (\cite{cai2014two}), CQ2010 (\cite{chen2010two}), CLZ2014 (\cite{chen2014two,chen2019two}) and SD2008 (\cite{srivastava2008test}) tests. All these tests are implemented from the `highmean' package in R (\cite{highmean}).

For the combinations of the three models and the three values of $p$, we first estimate the sizes of all the tests based on 1000 independent replications at the nominal level of 5\%. The estimated sizes are reported in \autoref{tab:size1}.

\begin{table}[h]
\centering
\caption{Estimated sizes at nominal level 5\% for the different tests for high-dimensional data based on 1000 independent replications for $n_1 = 40$ and $n_2 = 50$.}
\label{tab:size1}
\begin{tabular}{ccccccccccc}
\hline
Model   &$p$   &ZGZC2020   &SS   &BS1996   &CLX2014   &CQ2010   &CLZ2014   &SD2008\\\hline 
 1   &30   &0.049   &0.045   &0.069   &0.041   &0.070   &0.128   &0.031\\ 
 1   &50   &0.046   &0.044   &0.063   &0.029   &0.065   &0.155   &0.026\\ 
 1   &100   &0.054   &0.050   &0.072   &0.040   &0.070   &0.195   &0.019\\ 
 2   &30   &0.056   &0.049   &0.074   &0.032   &0.073   &0.160   &0.026\\ 
 2   &50   &0.063   &0.059   &0.074   &0.034   &0.073   &0.156   &0.032\\ 
 2   &100   &0.051   &0.051   &0.061   &0.023   &0.062   &0.185   &0.020\\ 
 3   &30   &0.010   &0.047   &0.017   &0.001   &0.017   &0.116   &0.006\\ 
 3   &50   &0.015   &0.057   &0.025   &0.000   &0.021   &0.154   &0.002\\ 
 3   &100   &0.016   &0.050   &0.017   &0.000   &0.017   &0.214   &0.000\\\hline
\end{tabular}
\end{table}

It can be observed in \autoref{tab:size1} that for all the models and values of $p$, the estimated sizes of the SS test are close to the nominal level of 5\%. The estimated sizes of the ZGZC2020 test are close to the nominal level in Models 1 and 2, but significantly smaller than the nominal level for the non-Gaussian distribution in Model 3. The estimated sizes of the BS1996 test and the CQ2010 test are usually slightly larger than the nominal level of 5\% in the majority of cases in Models 1 and 2, but quite small compared to the nominal level in Model 3. The estimated sizes of the CLX2014 test are usually not far from the nominal level except in Model 3, where it is almost 0. The estimated sizes of the SD2008 test are somewhat smaller than the nominal level throughout Models 1 and 2, and become  close to zero in Model 3. The estimated sizes of the CLZ2014 test are always very higher than the nominal level irrespective of the models and values of $p$, its lowest estimated size being almost 12\%.

\begin{figure}
	\centering
	\includegraphics[width=1\textwidth]{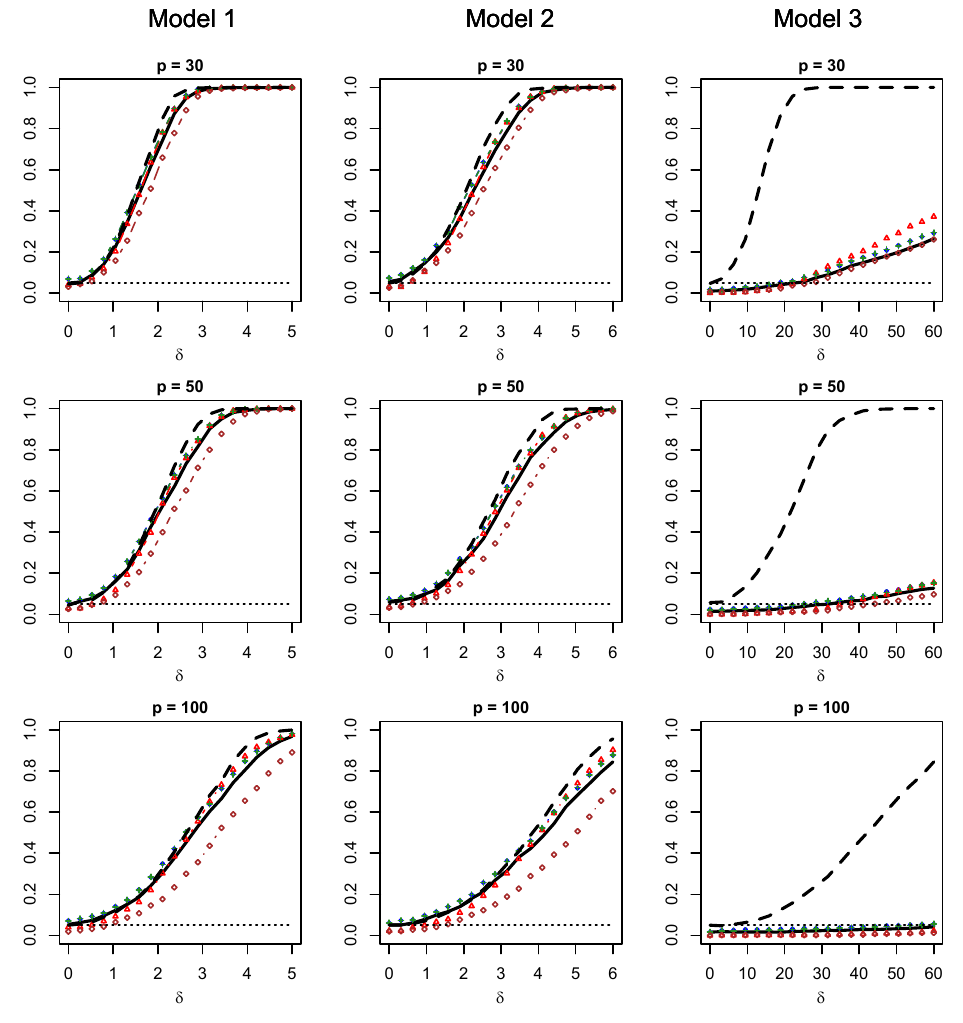}
	\caption{Estimated powers at nominal level 5\% (`$\cdots$') for the ZGZC2020 test (`$\boldsymbol{\leftrightline}$'), the SS test (`$\boldsymbol{-}\boldsymbol{-}$'), the BS1996 test (`\textcolor{blue}{$-\circ-$}'), the CLX2014 test (`\textcolor{red}{$-\smalltriangleup-$}'), the CQ2010 test (`\textcolor{ForestGreen}{$-+-$}') and the SD2008 test (`\textcolor{brown}{$-\smalldiamond-$}') based on 1000 independent replications for $n_1 = 40$ and $n_2 = 50$.}
	\label{fig:fig1}
\end{figure}

The above observations indicate that the SS test has satisfactory sizes irrespective of Gaussianity or non-Gaussianity or values of $p$. The ZGZC2020 test has satisfactory sizes in the Gaussian distribution or the non-Gaussian distribution when the moment assumptions of its validity are satisfied, but it fails to have proper size in case the sample arises from a non-Gaussian distribution not satisfying the moment conditions of its theoretical validity. The estimated sizes of the other tests are generally not so close to the nominal level compared to those of the ZGZC2020 test and the SS test.

Next, we proceed to demonstrate the estimated  powers of the tests. Since the CLZ2014 test is so improperly sized, we drop it from the power comparison. As mentioned before, by varying the value of $\delta$ within a range from 0 to a positive number, we generate estimated power curves of the tests. For the nine total combinations of the three models and three values of $p$, we get in total nine plots, which we present together in \autoref{fig:fig1}. Each column in \autoref{fig:fig1} corresponds to one of the models, which is mentioned at the top of the figure, and the values of $p$ for each of the plot are mentioned in the headings of the individual plots. Within each plot, the estimated power curves of the tests are plotted against the values of $\delta$.

From the plots in \autoref{fig:fig1}, it can be observed that the estimated power curve of the SS test is almost uniformly higher than all the other tests except the BS1996 test and the CQ2010 test for small to moderate values of $\delta$, irrespective of Gaussian or non-Gaussian distributions, although for the Gaussian distribution, the difference between the estimated powers of the SS test and the best performing tests among the rest is quite narrow. For small to moderate values of $\delta$ too, the estimated powers of the SS test is very close to those of the BS1996 test and the CQ2010 test. Also, similar to what was observed in \autoref{tab:size1}, for $\delta = 0$, which corresponds to the null hypothesis being true, the estimated powers of the BS1996 test and the CQ2010 test are  slightly but noticeably higher than the nominal level of 5\%. This perhaps contributes to the head-start of their power curves against the SS test over $\delta$, which vanishes gradually as $\delta$ becomes larger. The difference between the power curves of the SS test and the other tests increases in the non-Gaussian distribution in Model 2. However, in Model 3, the difference between the powers of the SS test and all other tests is quite stark. The ZGZC2020 test also performs well in Models 1 and 2, but not in Model 3. The powers of all the tests decrease as $p$ increases, and the difference between the power of the SS test and the other tests also increases with $p$.

The overall finding from the plots in \autoref{fig:fig1} indicates the utility of the SS test irrespective of the Gaussianity or non-Gaussianity of the underlying distributions. Whether the magnitude of $p$ is similar to the sample size or considerably larger, the SS test is observed to perform well and almost uniformly beating all the other tests if one accounts for the proper size of the tests at the nominal level. Moreover, for a non-Gaussian distribution where the moment conditions for the validity of other tests do not hold, the difference between the performances of the SS tests and all the other tests is profound.

\subsection{Comparison with multivariate tests}\label{subsec:simulation_scaling}
Recall that the ZGZC2020 test and the SS test do not normalize the test statistics with the respective inverses of the corresponding sample covariance matrices. Tests for multivariate data almost always involve normalizing with some sample covariance matrices, which makes the tests scale invariant, and as a further consequence the null distributions of those tests are usually free of any population parameters except the dimension, and the sample size for non-asymptotic tests. Evidently, the ZGZC2020 test and the SS test are not scale invariant. However, it is of interest to investigate whether this affects their performance compared to their normalized counterparts in multivariate data. This would be further of interest because to propose these tests to be employed irrespective of the dimension of the sample, for both small values of $p$ and large values of $p$, it would be required of them to have satisfactory performance for small values of $p$ too. We investigated their performance for large $p$ in the previous subsection, and here we compare their performances against the Hotelling's $T^2$ test and the test by \cite{choi1997approach} for small $p$. We denote the Hotelling's $T^2$ test as the HT2 test and the test by \cite{choi1997approach} as the CM1997 test.

Here again, we consider the same three cases of $\nu_{p, 0}$ as in \autoref{subsec:simulation_highdim} and like before, we take $n_1 = 40$ and $n_2 = 50$. However, we now fix $p = 2$. So, here, in Model 1, $\nu_{p, 0}$ is a bivariate Gaussian distribution with zero mean and covariance matrix $\bSigma_2 = (( \sigma_{i j} ))$ given by $\sigma_{i j} = 0.5 + 0.5 I(i = j)$. In Model 2 and Model 3, $\nu_{p, 0}$ is the corresponding bivariate $t_4$ distribution and the bivariate Cauchy distribution, respectively. However, we change the value of the vector $\bh_p \equiv \bh_2$ and consider two cases as given below:
\begin{itemize}
    \item \textit{Case 1}: $\bh_2 = (1, 1)^\top$,
    \item \textit{Case 2}: $\bh_2 = (0, 1)^\top$.
\end{itemize}
For each of these two cases, we choose the shifts $\ba_{2, 1} = \mathbf{0}_2$ and $\ba_{2, 2} = \delta \bh_2$, where $\delta$ is a real number, like in \autoref{subsec:simulation_highdim}. The value $\delta = 0$ corresponds to the null hypothesis \eqref{h_0} being true. Like before, we take the values of $\delta$ on a grid between 0 and some suitable positive numbers to generate estimated power curves of the tests.

\begin{table}[h]
\centering
\caption{Estimated sizes at nominal level 5\% for the different tests for multivariate data based on 1000 independent replications for $n_1 = 40$ and $n_2 = 50$.}
\label{tab:size2}
\begin{tabular}{ccccc}
\hline
Model   &ZGZC2020   &HT2   &SS   &CM1997\\\hline 
 1   &0.060   &0.049   &0.057   &0.049\\ 
 2   &0.054   &0.050   &0.060   &0.049\\ 
 3   &0.018   &0.012   &0.054   &0.048\\\hline
\end{tabular}
\end{table}

In \autoref{tab:size2}, the estimated sizes of the ZGZC2020, HT2, SS and CM1997 tests are presented for all three models. It should be noted that here the HT2 test is an exact non-asymptotic test, while the other three tests are asymptotic tests. It can be seen from this table that the estimated sizes for all the tests are close to the nominal level in Models 1 and 2. However, in Model 3, the estimated sizes of the ZGZC2020 test and the HT2 test are considerably different from the nominal level, yet the estimated sizes of the SS test and the CM1997 test are again close to the nominal level. This indicates that significant departure from Gaussianity similarly affects the ZGZC2020 test and the HT2 test, while the SS test and the CM1997 test remain comparably unaffected.

Next, we present the estimated power curves in \autoref{fig:fig2}. Here, six plots are presented accounting for the six combinations of the two cases of the shifts and the three models. Each column of plots in \autoref{fig:fig2} corresponds to one of the models, mentioned at the top of the respective columns, while each row corresponds to one of the cases as depicted on the left sides of the corresponding columns.

\begin{figure}[h]
	\centering
	\includegraphics[width=1\textwidth]{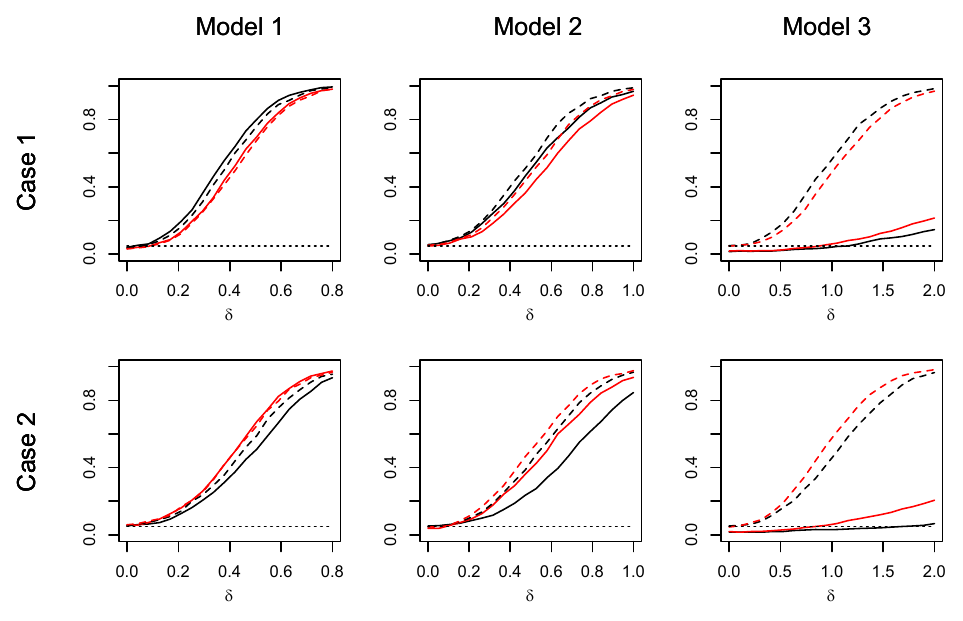}
	\caption{Estimated powers at nominal level 5\% (`$\cdots$') for the ZGZC2020 test (`$\leftrightline$'), the HT2 test (`\textcolor{red}{$\leftrightline$}'), the SS test (`$--$') and the CM1997 test (`\textcolor{red}{$--$}') based on 1000 independent replications for $n_1 = 40$ and $n_2 = 50$.}
	\label{fig:fig2}
\end{figure}
In \autoref{fig:fig2}, it can be observed that for Case 1, the un-normalized tests ZGZC2020 and SS have higher estimated power than their normalized counterparts, namely, the HT2 test and the CM1997 test, uniformly over the range of $\delta$. In Model 1 in Case 1, although the underlying distribution is Gaussian, both the ZGZC2020 test and the SS test have higher estimated power than the HT2 test. This is in spite of the fact that the HT2 test is an exact non-asymptotic test for Gaussian distributions and it is the most powerful invariant test. In Case 2, however, all the normalized tests have uniformly higher estimated power compared to the un-normalized counterparts. In both Case 1 and Case 2, the estimated power curve of the SS test is either higher or coincides with that of the ZGZC2020 test except in Model 1 in Case 1, where the ZGZC2020 test has narrowly higher power than the SS test.

\subsection{Analysis of real data}\label{subsec:realdata}
We analyze the colon data available at \url{http://genomics-pubs.princeton.edu/oncology/affydata/index.html}. This dataset was first analyzed in \cite{alon1999broad} and later in many other works including \cite{zhang2020simple}. It contains the expression values of 2000 genes from 62 samples of human colon tissue, 22 are from normal tissue and 40 are from tumor tissue. We are interested in testing whether the gene expression values differ between normal and tumor tissue. So, we have $p = 2000$, and $n_1 = 40, n_2 = 22$. We consider the ZGZC2020 test, the SS test, the BS1996 test, the CLX2014 test, the CQ2010 test and the SD2008 test. The p-values obtained are presented in the first row of \autoref{tab:pvalues}. It can be observed that several p-values are quite small, but even among them, the p-value for the SS test is the closest to zero.

\begin{table}[h]
\centering
\caption{First row: p-values in the colon data with $p = 2000$.\\
Second row: Average of p-values in the colon data with $p = 40$, p-values are averages over 50 blocks.}
\label{tab:pvalues}
\begin{tabular}{cccccc}
\hline
ZGZC2020           &SS      &BS1996      &CLX2014     &CQ2010      &SD2008 \\\hline
6.259680e-04 &0.000000e+00 &4.002116e-07 &9.264269e-07 &1.273052e-07 &2.702844e-01 \\\hline
0.03967578 &0.01588000 &0.03919005 &0.08038538 &0.03124023 &0.15816933 \\\hline
\end{tabular}
\end{table}

\begin{figure}[h]
	\centering
	\includegraphics[width=1\textwidth]{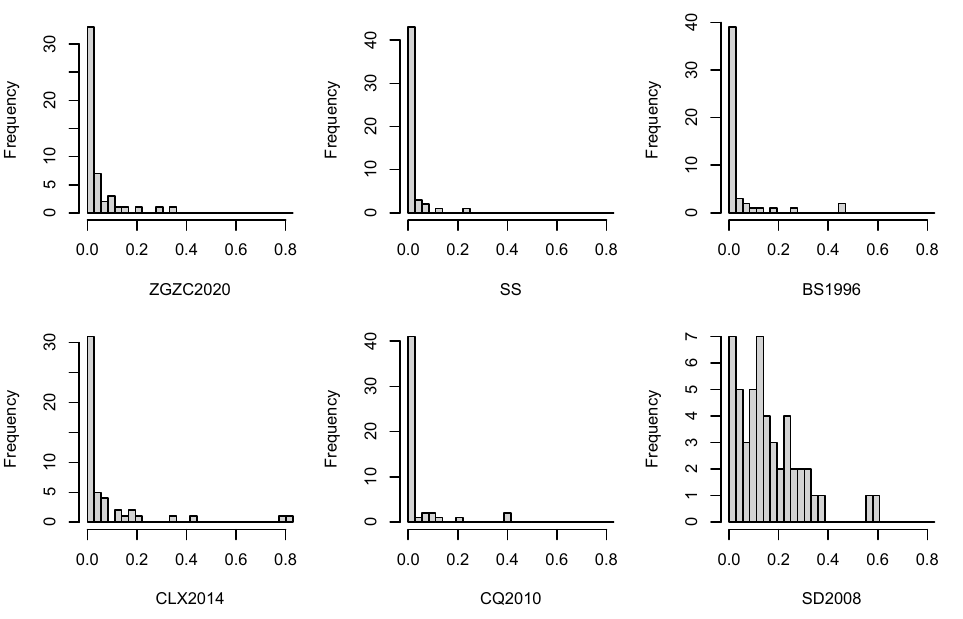}
	\caption{Histograms of p-values of the tests in the colon data.}
    \label{fig:fig3}
\end{figure}

Next, we are interested in observing the p-values of the tests when the dimension $p$ is closer in magnitude to either of the sample sizes $n_1$ and $n_2$. To implement this, we partition the data matrix by taking blocks of 40 consecutive columns. This way, we have 50 data matrices each with $p = 40$, $n_1 = 40$ and $n_2 = 22$. So, here the magnitude of $p$ is comparable to the sample size.
We present the histogram of the 50 p-values for each of the tests in \autoref{fig:fig3}. The averages of the p-values are presented in the second row of \autoref{tab:pvalues}. Although the p-value of the SS test was the smallest in the first row of \autoref{tab:pvalues}, most of the other p-values were also nearly zero, and it was difficult to conclude that one of the tests has higher statistical power compared to the other. However, as the value of $p$ becomes lower and comparable to the magnitude of $n_1$ and $n_2$ in the study presented in the second row of \autoref{tab:pvalues}, the p-values of all the tests increase. In that case, from the second row of \autoref{tab:pvalues}, it can be seen that the SS test has the lowest average p-value. In \autoref{fig:fig3} too, it can be observed that the p-values of the SS test are concentrated closest to zero. This indicates that the SS test has indeed higher statistical power compared to the other tests carried out on this data.

\section{Concluding remarks}\label{sec:conclusion}
In this work, we have proposed a new notion of convergence of fixed-dimensional functions of high-dimensional random vectors, which is uniform over the dimension of the aforementioned random vectors. This notion of convergence can be employed to derive asymptotic null distributions of tests, which are applicable for both low-dimensional and high-dimensional data. We have utilized this notion of convergence to derive the asymptotic null distribution of a general class of kernel-based multi-sample tests of equality of locations. We have then considered two specific choices of the kernel, one of which yielded an existing test in the literature (denoted as the ZGZC2020 test above, proposed by \cite{zhang2020simple}) and the other one gave a spatial sign-based test (denoted as the SS test). The notion of uniform-over-dimension convergence to the asymptotic null distribution is expected to yield tests which perform as good for high-dimensional data as for low-dimensional data, and we have demonstrated this fact using simulated and real data. We have shown that the SS test in particular outperforms several popular tests in the literature for high-dimensional data in both Gaussian and non-Gaussian models. We have also provided examples of low-dimensional non-Gaussian and even Gaussian models, where the SS test outperforms even the Hotelling's $T^2$ test.

The applicability of the notion of uniform-over-dimension convergence on some of the common results in the usual convergence in distribution and probability is demonstrated in this work. However, many other useful results remain to be explored, e.g., the Lindeberg central limit theorem and U-statistic asymptotic results, which we have not considered in this work and intend to study subsequently. It is of interest too to study other asymptotic tests for multivariate and high-dimensional data, whether they can be implemented uniformly-over-dimension, so that they yield good performance irrespective of the dimension of the data.

%\bibliographystyle{apalike}
%\bibliography{bibliography}

%%%%%%%%%%%%%%%%%%%%%%%%%%%%%%%%%%%%%%%%%%%%%%
%% Single Appendix:                         %%
%%%%%%%%%%%%%%%%%%%%%%%%%%%%%%%%%%%%%%%%%%%%%%
\begin{appendix}
\section*{Proofs of mathematical results}%% if no title is needed, leave empty \section*{}.
\label{appendix}

%Proof of \autoref{lemma1} is placed in the Supplement \cite{chowdhury2024supplement}.
\begin{proof}[Proof of \autoref{lemma1}]
For every $n$, define $g_n, h_n : \mathbb{R}^d \to \mathbb{R}$ by
\begin{align*}
& g_n( \bx ) = \inf\{ I_A( \by ) + n \| \by - \bx \| : \by \in \mathbb{R}^d \} , \\
& h_n( \bx ) = \sup\{ I_A( \by ) - n \| \by - \bx \| : \by \in \mathbb{R}^d \} .
\end{align*}
For every $\bx$ and every $n$, we have
\begin{align*}
g_n(\bx) & \le g_{n+1}(\bx) \le I_A(\bx) \le 1,\\
g_n(\bx) & \ge \inf\{ I_A( \by ) : \by \in \mathbb{R}^d \} + \inf\{ n \| \by - \bx \| : \by \in \mathbb{R}^d \} \\
& \ge \inf\{ I_A( \by ) : \by \in \mathbb{R}^d \} = 0 ,
\intertext{and}
h_n(\bx) & \ge h_{n+1}(\bx) \ge I_A(\bx) \ge 0,\\
h_n(\bx) & \le \sup\{ I_A( \by ) : \by \in \mathbb{R}^d \} + \sup\{ - n \| \by - \bx \| : \by \in \mathbb{R}^d \} \\
& \le \sup\{ I_A( \by ) : \by \in \mathbb{R}^d \} = 1 .
\end{align*}
So, $g_n$ and $h_n$ are bounded functions for every $n$ with
\begin{align}
0 \le g_n(\bx) \le I_A(\bx) \le h_n(\bx) \le 1 .
\label{lemma1_eq1}
\end{align}
Also, $g_n$ and $h_n$ are Lipschitz continuous functions, since
\begin{align*}
g_n( \bx ) & \le \inf\{ I_A( \by ) + n \| \by - \bz \| : \by \in \mathbb{R}^d \} + n \| \bz - \bx \| \\
& = g_n( \bz ) + n \| \bz - \bx \| ,\\
h_n( \bx ) & \ge \sup\{ I_A( \by ) - n \| \by - \bz \| : \by \in \mathbb{R}^d \} - n \| \bz - \bx \| \\
& = h_n( \bz ) - n \| \bz - \bx \| ,
\end{align*}
which implies
\begin{align}
& | g_n( \bx ) - g_n( \bz ) | \le n \| \bz - \bx \|
\text{ and }
| h_n( \bx ) - h_n( \bz ) | \le n \| \bz - \bx \| .
\label{lemma1_eq2}
\end{align}

Next, for a set $B \subset \mathbb{R}^d$ and a number $\delta > 0$, define $B^\delta = \{ \by \in \mathbb{R}^d : \inf_{\bx \in B} \| \by - \bx \| \le \delta \}$. We shall establish that
\begin{align}
\text{for any } \bx \text{ with } \bx \notin A^{1/n} \cap (A^c)^{1/n},\; g_n(\bx) = I_A(\bx) = h_n(\bx) .
\label{lemma1_eq3}
\end{align}
Note that $( A^{1/n} )^c \subset A^c$ and $( (A^c)^{1/n} )^c \subset A$.
For $\bx \in A^c$, from \eqref{lemma1_eq1} we have $0 \le g_n(\bx) \le I_A(\bx) = 0$, so $g_n(\bx) = I_A(\bx)$. Similarly, for $\bx \in A$, $h_n(\bx) = I_A(\bx)$.
Now, suppose if possible, for some $\bx \in ( (A^c)^{1/n} )^c$, $g_n(\bx) < I_A(\bx)$. This implies from the definition of $g_n(\bx)$ that there exists $\by$ with $I_A(\by) + n \| \by - \bx \| < I_A(\bx)$. Since $n \| \by - \bx \| \ge 0$, we must have $I_A(\by) = 0$ for the strict inequality to hold, which implies $n \| \by - \bx \| < I_A(\bx) = 1$ because $\bx \in ( (A^c)^{1/n} )^c \subset A$, and this yields $\| \by - \bx \| < 1/n$. Now $I_A(\by) = 0$ and $\| \by - \bx \| < 1/n$ together imply that $\bx \in (A^c)^{1/n}$, a contradiction. So, we must have $g_n(\bx) = I_A(\bx)$ for every $\bx \in ( (A^c)^{1/n} )^c$. Similarly, it follows that $h_n(\bx) = I_A(\bx)$ for every $\bx \in ( A^{1/n} )^c$. Therefore, we have
\begin{align}
\begin{split}
& g_n(\bx) = I_A(\bx) \text{ for } \bx \in ( (A^c)^{1/n} )^c \cup A^c ,\\
& h_n(\bx) = I_A(\bx) \text{ for } \bx \in A \cup ( A^{1/n} )^c .
\end{split}
\label{lemma1_eq4}
\end{align}
So from \eqref{lemma1_eq4}, we have $g_n(\bx) = I_A(\bx) = h_n(\bx)$ for every $\bx \in ( ( (A^c)^{1/n} )^c \cup A^c ) \cap ( A \cup ( A^{1/n} )^c ) = ( (A^c)^{1/n} )^c \cup ( A^{1/n} )^c = ( A^{1/n} \cap (A^c)^{1/n} )^c$, which establishes \eqref{lemma1_eq3}. Therefore, from \eqref{lemma1_eq3} we have, for every $n$ and any probability measure $\gamma$ on $( \mathbb{R}^d, \mathcal{R}^d )$,
\begin{align}
\int (h_n - g_n) \mathrm{d} \gamma < \gamma\big( A^{1/n} \cap (A^c)^{1/n} \big) .
\label{lemma1_eq5}
\end{align}

Next, we proceed to show that
\begin{align}
\sup_p \mu_p\big( A^{1/n} \cap (A^c)^{1/n} \big) \to 0 \text{ as } n \to \infty .
\label{lemma1_eq6}
\end{align}
Suppose, if possible, \eqref{lemma1_eq6} is not true, which means there exists $\epsilon > 0$ and a subsequence $\{n_k\}$ such that for every $k$,
\begin{align*}
\sup_p \mu_p\big( A^{1/n_k} \cap (A^c)^{1/n_k} \big) > \epsilon .
\end{align*}
Then, there is a sequence $\{\mu_{p_k}\}$ corresponding to $\{n_k\}$ such that for every $k$,
\begin{align}
\mu_{p_k}\big( A^{1/n_k} \cap (A^c)^{1/n_k} \big) > \epsilon .
\label{lemma1_eq7}
\end{align}
From assumption \ref{assumption1}, $\{\mu_{p_k}\}$ has a subsequence $\{\mu_{p_l}\}$ such that
\begin{align}
\sup_{C \in \mathcal{R}^d} \left| \mu_{p_{n_l}}( C ) - \nu( C ) \right| \to 0
\text{ as } l \to \infty
\label{lemma1_eq8}
\end{align}
for some probability measure $\nu$ on $( \mathbb{R}^d, \mathcal{R}^d )$, and so
\begin{align}
& \left| \mu_{p_{n_l}}\big( A^{1/n_l} \cap (A^c)^{1/n_l} \big) - \nu\big( A^{1/n_l} \cap (A^c)^{1/n_l} \big) \right| \nonumber\\
& \le \sup_{C \in \mathcal{R}^d} \left| \mu_{p_{n_l}}( C ) - \nu( C ) \right| \to 0
\text{ as } l \to \infty
\label{lemma1_eq9}
\end{align}
From \eqref{lemma1_eq7}, \eqref{lemma1_eq8} and \eqref{lemma1_eq9}, we have for all sufficiently large $l$,
\begin{align}
\nu\big( A^{1/n_l} \cap (A^c)^{1/n_l} \big) > \epsilon / 2 .
\label{lemma1_eq10}
\end{align}
Now, $\{ A^{1/n} \cap (A^c)^{1/n} \}$ is a decreasing sequence of sets with $A^{1/n} \cap (A^c)^{1/n} \to \partial A$ as $n \to \infty$, which implies
\begin{align}
\nu\big( A^{1/n_l} \cap (A^c)^{1/n_l} \big) \to \nu(\partial A) \text{ as } l \to \infty .
\label{lemma1_eq11}
\end{align}
From \eqref{lemma1_eq10} and \eqref{lemma1_eq11}, we have for all sufficiently large $l$,
\begin{align}
\nu(\partial A) > \epsilon / 4 .
\label{lemma1_eq12}
\end{align}
Now, from \eqref{lemma1_eq8} we have
\begin{align}
& \left| \mu_{p_{n_l}}(\partial A) - \nu(\partial A) \right| 
\le \sup_{C \in \mathcal{R}^d} \left| \mu_{p_{n_l}}( C ) - \nu( C ) \right| \to 0
\text{ as } l \to \infty ,
\label{lemma1_eq13}
\end{align}
and \eqref{lemma1_eq12} and \eqref{lemma1_eq13} together imply that for all sufficiently large $l$,
\begin{align*}
\mu_{p_{n_l}}(\partial A) > \epsilon / 8 .
\end{align*}
But this is a direct contradiction to the fact that $A$ is a $\mu_p$-continuity set for all $p$, which means $\mu_{p_{n_l}}(\partial A) = 0$ for every $l$.
This establishes the validity of \eqref{lemma1_eq6}.

Hence, from \eqref{lemma1_eq5} and \eqref{lemma1_eq6} we have
\begin{align*}
\sup_p \int (h_n - g_n) \mathrm{d} \mu_p < \sup_p \mu_p\big( A^{1/n} \cap (A^c)^{1/n} \big) \to 0 \text{ as } n \to \infty .
\end{align*}
Given any $\epsilon > 0$, choose $n$ such that $\sup_p \mu_p\big( A^{1/n} \cap (A^c)^{1/n} \big) < \epsilon$.
The corresponding $g_n$ and $h_n$ functions would satisfy the roles of the functions $g$ and $h$ in the statement of the lemma.
\end{proof}

% Proof of \autoref{portmanteau} is placed in the Supplement \cite{chowdhury2024supplement}.
\begin{proof}[Proof of \autoref{portmanteau}]
We shall prove in the following sequence: \ref{p1} $\implies$ \ref{p2} $\implies$ \ref{p3} $\implies$ \ref{p4} $\implies$ \ref{p5} $\implies$ \ref{p1} and \ref{p5} $\implies$ \ref{p6} $\implies$ \ref{p3}.

\ref{p1} $\implies$ \ref{p2} is obvious from \autoref{definition1}.

\ref{p2} $\implies$ \ref{p3}:
It is enough to prove that every continuous function $ f : \mathbb{R}^d \to \mathbb{R} $, which is zero outside of a compact set $K$, is bounded and uniformly continuous.

Because $f$ is continuous on $K$ and 0 outside $K$, it is bounded on $\mathbb{R}^d$ (Theorem 4.15 in \cite{rudin1976principles}). Next, suppose if possible, $f$ is not uniformly continuous on $\mathbb{R}^d$. Then, there is $\epsilon > 0$ and sequences $\{ \ba_n \}$ and $\{ \bb_n \}$ such that $\| \ba_n - \bb_n \| \to 0$ as $n \to \infty$ but $| f(\ba_n) - f(\bb_n) | > \epsilon$ for every $n$. Consider the sets $A = \{ \ba_n \} \cap K$ and $B = \{ \bb_n \} \cap K$. If $A$ has infinitely many elements, then there is a subsequence $\{\ba_{n_i}\}$ of $\{\ba_n\}$ and a point $\ba \in K$ with $\| \ba_{n_i} - \ba \| \to 0$ as $n_i \to \infty$ (Theorem 2.37 in \cite{rudin1976principles}), otherwise, there is $N_1$ such that for all $n \ge N_1$, $\ba_n \in K^c$ with $f(\ba_n) = 0$. Similarly, if $B$ has infinitely many elements, then there is a subsequence $\{\bb_{n_j}\}$ of $\{\bb_n\}$ and a point $\bb \in K$ with $\| \bb_{n_j} - \bb \| \to 0$ as $n_j \to \infty$, and otherwise, there is $N_2$ such that for all $n \ge N_2$, $\bb_n \in K^c$ with $f(\bb_n) = 0$.

Suppose $A$ has infinitely many elements. Then, $\| \ba - \bb_{n_i} \| \le \| \ba - \ba_{n_i} \| + \| \ba_{n_i} - \bb_{n_i} \| \to 0$ as $n_i \to \infty$, and $| f(\ba) - f(\ba_{n_i}) | \to 0$ as $n_i \to \infty$ from the continuity of $f$. However, this yields $| f(\ba) - f(\bb_{n_i}) | \ge | f(\ba_{n_i}) - f(\bb_{n_i}) | - | f(\ba) - f(\ba_{n_i}) | \ge \epsilon / 2$ for all sufficiently large $n$, which contradicts the continuity of $f$ at $\ba$. Now suppose $B$ has infinitely many elements. Using the same arguments as above, we again get a contradiction to the continuity of $f$ at $\bb$. Finally, suppose $A$ and $B$ have both finitely many elements, which means for all $n \ge \max\{ N_1, N_2 \}$, $\ba_n, \bb_n \in K^c$ with $f(\ba_n) = f(\bb_n) = 0$, contradicting the statement that $| f(\ba_n) - f(\bb_n) | > \epsilon$ for every $n$.
Therefore, $f$ must be uniformly continuous on $\mathbb{R}^d$.

\ref{p3} $\implies$ \ref{p4}:
Let $f : \mathbb{R}^d \to \mathbb{R}$ be a bounded and Lipschitz continuous function, and let $M = \sup\{ | f( \bx ) | : \bx \in \mathbb{R}^d \}$.
Given $\epsilon > 0$, under assumption \ref{assumption1}, there is a compact set $K$ such that \eqref{tightness} holds: $\mu_p(K) > 1 - (\epsilon / 2 M)$ for all $p$. Define $h : \mathbb{R}^d \to \mathbb{R}$ as
\begin{align*}
h( \bx ) = \sup\{ I_K( \by ) - \| \by - \bx \| : \by \in \mathbb{R}^d \} .
\end{align*}
Note that $h$ is the same function $h_n$ defined in the proof of \autoref{lemma1} with $A$ replaced by the compact set $K$ and $n = 1$. So, from \eqref{lemma1_eq1} and \eqref{lemma1_eq2} in the proof of \autoref{lemma1}, it follows that $0 \le h( \cdot ) \le 1$ and $h$ is a Lipschitz continuous function. Also, from \eqref{lemma1_eq4}, it follows that $h( \bx ) = 1$ for $\bx \in K$ and $h( \bx ) = 0$ for $\bx \notin K^1$, where $K^1 = \{ \by \in \mathbb{R}^d : \inf_{\bx \in K} \| \by - \bx \| \le 1 \}$. Note that $K \subset K^1$. Summarising, we have
\begin{align}
\begin{split}
& h \text{ is continuous},\\
& h( \bx ) =
\begin{cases}
1 & \text{for } \bx \in K ,\\
0 & \text{for } \bx \notin K^1 ,
\end{cases}\\
& \text{and } 0 \le h( \cdot ) \le 1 \text{ for } \bx \in K^1 \cap K^c .
\end{split}
\label{portmanteau_eq1}
\end{align}
So, from \eqref{portmanteau_eq1}, it follows that
\begin{align}
1 \ge \int h( \bx ) \mathrm{d} \mu_p( \bx ) \ge \mu_p(K) > 1 - (\epsilon / 2 M) \text{ for all } p .
\label{portmanteau_eq2}
\end{align}
Suppose $\by \in (K^1)^c$. Then, $\inf_{\bx \in K} \| \by - \bx \| > 1 + \delta$ for some $\delta > 0$. For $\bz$ with $\| \by - \bz \| < \delta$, $\| \bz - \bx \| \ge \| \by - \bx \| - \| \by - \bz \|$ for any $\bx$, which implies $\inf_{\bx \in K} \| \bz - \bx \| \ge \inf_{\bx \in K} \| \by - \bx \| - \| \by - \bz \| > 1$, and hence $\bz \in (K^1)^c$. Therefore, $(K^1)^c$ is an open set, which implies that $K^1$ is closed. Further, since $K$ is a compact set, it is bounded, and hence $K^1$ is bounded. So, $K^1$, being a closed and bounded subset of $\mathbb{R}^d$, is a compact set (Theorem 2.41 in \cite{rudin1976principles}).

Since $K^1$ is compact, \eqref{portmanteau_eq1} implies that $h(\cdot)$ and $f(\cdot) h(\cdot)$ are continuous functions which are 0 outside the compact set $K^1$, and so from \ref{p3} we have for all sufficiently large $n$,
\begin{align}
\begin{split}
& \sup_p \left| \int h( \bx ) \mathrm{d} \mu_{n, p}( \bx ) - \int h( \bx ) \mathrm{d} \mu_p( \bx ) \right| < \epsilon / M ,\\
& \sup_p \left| \int f( \bx ) h( \bx ) \mathrm{d} \mu_{n, p}( \bx ) - \int f( \bx ) h( \bx ) \mathrm{d} \mu_p( \bx ) \right| < \epsilon .
\end{split}
\label{portmanteau_eq3}
\end{align}
Note that
\begin{align}
\begin{split}
& \left| \int f( \bx ) \mathrm{d} \mu_{n, p}( \bx ) - \int f( \bx ) \mathrm{d} \mu_p( \bx ) \right| \\
& \le \left| \int f( \bx ) ( 1 - h( \bx ) ) \mathrm{d} \mu_{n, p}( \bx ) \right| \\
& \quad + \left| \int f( \bx ) h( \bx ) \mathrm{d} \mu_{n, p}( \bx ) - \int f( \bx ) h( \bx ) \mathrm{d} \mu_p( \bx ) \right| \\
& \quad + \left| \int f( \bx ) ( 1 - h( \bx ) ) \mathrm{d} \mu_p( \bx ) \right| .
\end{split}
\label{portmanteau_eq4}
\end{align}
Now,
\begin{align}
\left| \int f( \bx ) ( 1 - h( \bx ) ) \mathrm{d} \mu_{n, p}( \bx ) \right|
& \le M \int ( 1 - h( \bx ) ) \mathrm{d} \mu_{n, p}( \bx ) \nonumber\\
& = M \left( 1 - \int h( \bx ) \mathrm{d} \mu_{n, p}( \bx ) \right) ,
\label{portmanteau_eq5}
\end{align}
and similarly,
\begin{align}
\left| \int f( \bx ) ( 1 - h( \bx ) ) \mathrm{d} \mu_p( \bx ) \right|
& \le M \left( 1 - \int h( \bx ) \mathrm{d} \mu_p( \bx ) \right) .
\label{portmanteau_eq6}
\end{align}
So, \eqref{portmanteau_eq4}, \eqref{portmanteau_eq5} and \eqref{portmanteau_eq6} together yield
\begin{align}
& \left| \int f( \bx ) \mathrm{d} \mu_{n, p}( \bx ) - \int f( \bx ) \mathrm{d} \mu_p( \bx ) \right| \nonumber\\
& \le M \left( 1 - \int h( \bx ) \mathrm{d} \mu_{n, p}( \bx ) \right) 
+ M \left( 1 - \int h( \bx ) \mathrm{d} \mu_p( \bx ) \right) \nonumber\\
& \quad + \left| \int f( \bx ) h( \bx ) \mathrm{d} \mu_{n, p}( \bx ) - \int f( \bx ) h( \bx ) \mathrm{d} \mu_p( \bx ) \right| \nonumber\\
& = 2 M \left( 1 - \int h( \bx ) \mathrm{d} \mu_p( \bx ) \right) \nonumber\\
& \quad + M \left( \int h( \bx ) \mathrm{d} \mu_p( \bx ) - \int h( \bx ) \mathrm{d} \mu_{n, p}( \bx ) \right) \nonumber\\
& \quad + \left| \int f( \bx ) h( \bx ) \mathrm{d} \mu_{n, p}( \bx ) - \int f( \bx ) h( \bx ) \mathrm{d} \mu_p( \bx ) \right| \nonumber\\
& \le 2 M \left( 1 - \int h( \bx ) \mathrm{d} \mu_p( \bx ) \right) \nonumber\\
& \quad + M \sup_p \left| \int h( \bx ) \mathrm{d} \mu_{n, p}( \bx ) - \int h( \bx ) \mathrm{d} \mu_p( \bx ) \right| \nonumber\\
& \quad + \sup_p \left| \int f( \bx ) h( \bx ) \mathrm{d} \mu_{n, p}( \bx ) - \int f( \bx ) h( \bx ) \mathrm{d} \mu_p( \bx ) \right| .
\label{portmanteau_eq7}
\end{align}
Applying \eqref{portmanteau_eq2} and \eqref{portmanteau_eq3} in \eqref{portmanteau_eq7}, we get for all sufficiently large $n$,
\begin{align*}
& \left| \int f( \bx ) \mathrm{d} \mu_{n, p}( \bx ) - \int f( \bx ) \mathrm{d} \mu_p( \bx ) \right|
< 2 M \times (\epsilon / 2 M) + M \times (\epsilon / M) + \epsilon = 3 \epsilon \nonumber\\
\intertext{for all $p$, which implies}
& \sup_p \left| \int f( \bx ) \mathrm{d} \mu_{n, p}( \bx ) - \int f( \bx ) \mathrm{d} \mu_p( \bx ) \right| < 3 \epsilon \text{ for all sufficiently large } n .
\end{align*}
Since $\epsilon > 0$ is arbitrary, we get $\lim_{n \to \infty} \sup_p \left| \int f \mathrm{d} \mu_{n, p} - \int f \mathrm{d} \mu_p \right| = 0$.

\ref{p4} $\implies$ \ref{p5}:
Under assumption \ref{assumption1} and from \autoref{lemma1}, for a set $A$, which is a $\mu_p$-continuity set for all $p$ and given any $\epsilon > 0$, we can find bounded and Lipschitz continuous functions $g, h : \mathbb{R}^d \to \mathbb{R}$ such that $g \le I_A \le h$ and $\sup_p \int (h - g) \mathrm{d} \mu_p < \epsilon$.
Since $g, h$ are bounded and Lipschitz functions, from \ref{p4} we have for all sufficiently large $n$,
\begin{align}
\sup_p \left| \int g \mathrm{d} \mu_{n, p} - \int g \mathrm{d} \mu_p \right| < \epsilon \quad\text{and}\quad
\sup_p \left| \int h \mathrm{d} \mu_{n, p} - \int h \mathrm{d} \mu_p \right| < \epsilon .
\label{portmanteau_eq8}
\end{align}
Now,
\begin{align*}
& \mu_{n, p}( A ) - \mu_p( A )
\le \left( \int h \mathrm{d} \mu_{n, p} - \int h \mathrm{d} \mu_p \right) + \left( \int h \mathrm{d} \mu_p - \int g \mathrm{d} \mu_p \right) ,\\
& \mu_p( A ) - \mu_{n, p}( A )
\le \left( \int h \mathrm{d} \mu_p - \int g \mathrm{d} \mu_p \right) + \left( \int g \mathrm{d} \mu_p - \int g \mathrm{d} \mu_{n, p} \right) ,\\
\intertext{which implies}
& | \mu_{n, p}( A ) - \mu_p( A ) |
\le \left| \int g \mathrm{d} \mu_{n, p} - \int g \mathrm{d} \mu_p \right|
+ \left| \int h \mathrm{d} \mu_{n, p} - \int h \mathrm{d} \mu_p \right|
+ \int (h - g) \mathrm{d} \mu_p .
\end{align*}
Hence, from \eqref{portmanteau_eq8}, we have
\begin{align*}
& \sup_p | \mu_{n, p}( A ) - \mu_p( A ) | \\
& \le \sup_p \left| \int g \mathrm{d} \mu_{n, p} - \int g \mathrm{d} \mu_p \right|
+ \sup_p \left| \int h \mathrm{d} \mu_{n, p} - \int h \mathrm{d} \mu_p \right|
+ \sup_p \int (h - g) \mathrm{d} \mu_p \\
& < 3 \epsilon \quad \text{for all sufficiently large $n$.}
\end{align*}
The proof follows since $\epsilon > 0$ is arbitrary.

\ref{p5} $\implies$ \ref{p1}:
Let $g : \mathbb{R}^d \to \mathbb{R}$ be a non-negative bounded continuous function, and let $g(\cdot) \le L$. Then, for any probability measure $\nu$ on $\mathbb{R}^d$,
\begin{align*}
\int g(\bx )  \mathrm{d} \nu( \bx )
& = \int \int_{0}^\infty I( t < g( \bx ) ) \mathrm{d} t \; \mathrm{d} \nu( \bx ) 
= \int \int_{0}^{L} I( t < g( \bx ) ) \mathrm{d} t \; \mathrm{d} \nu( \bx ),
\end{align*}
which, by Fubini's theorem (Theorem 18.3 in \cite{billingsley2008probability}), implies that
\begin{align}
\int g(\bx )  \mathrm{d} \nu( \bx )
= \int_{0}^{L} \int I( t < g( \bx ) ) \mathrm{d} \nu( \bx ) \; \mathrm{d} t .
\label{portmanteau_eq9}
\end{align}
Since $g$ is continuous, $\partial \{ \bx \in \mathbb{R}^d : t < g( \bx ) \} \subset \{ \bx \in \mathbb{R}^d : g( \bx ) = t \}$. There can be only countably many values of $t$ such that $\mu_p\{ \bx \in \mathbb{R}^d : g( \bx ) = t \} > 0$ for at least one $p$. So, $\{ \bx \in \mathbb{R}^d : g( \bx ) = t \}$ is a $\mu_p$-continuity set for every $p$ except for countably many $t$, which means, from \ref{p5},
\begin{align}
\lim\limits_{n \to \infty} \sup_p \left| \int I( t < g( \bx ) ) \mathrm{d} \mu_{n, p}( \bx ) - \int I( t < g( \bx ) ) \mathrm{d} \mu_p( \bx ) \right| = 0 ,
\label{portmanteau_eq10}
\end{align}
except for countably many $t$.
Define $g_n : \mathbb{R} \to \mathbb{R}$ by
\begin{align*}
g_n( t ) = \sup_p \left| \int I( t < f( \bx ) ) \mathrm{d} \mu_{n, p}( \bx ) - \int I( t < f( \bx ) ) \mathrm{d} \mu_p( \bx ) \right| .
\end{align*}
Then, clearly, $0 \le g_n \le 1$ for every $n$, and from \eqref{portmanteau_eq10}, $g_n( t ) \to 0$ except for countably many $t$. Therefore, from the bounded convergence theorem (Theorem 16.5 in \cite{billingsley2008probability}), we have
\begin{align}
\int_{0}^{L} g_n( t ) \mathrm{d} t \to 0 \text{ as } n \to \infty .
\label{portmanteau_eq11}
\end{align}
On the other hand, for every $p$ and $n$,
\begin{align*}
\left| \int g \mathrm{d} \mu_{n, p} - \int g \mathrm{d} \mu_p \right| 
& = \left| \int_{0}^{L} \left\{ \int I( t < g( \bx ) ) \mathrm{d} \mu_{n, p}( \bx ) - \int I( t < g( \bx ) ) \mathrm{d} \mu_p( \bx ) \right\} \mathrm{d} t \right| \\
& \le \int_{0}^{L} \left| \int I( t < g( \bx ) ) \mathrm{d} \mu_{n, p}( \bx ) - \int I( t < g( \bx ) ) \mathrm{d} \mu_p( \bx ) \right| \mathrm{d} t \\
& \le \int_{0}^{L} \sup_p \left| \int I( t < g( \bx ) ) \mathrm{d} \mu_{n, p}( \bx ) - \int I( t < g( \bx ) ) \mathrm{d} \mu_p( \bx ) \right| \mathrm{d} t \\
& = \int_{0}^{L} g_n( t ) \mathrm{d} t ,
\end{align*}
which implies
\begin{align}
\sup_p \left| \int g \mathrm{d} \mu_{n, p} - \int g \mathrm{d} \mu_p \right| \le \int_{0}^{L} g_n( t ) \mathrm{d} t .
\label{portmanteau_eq12}
\end{align}
Hence, \eqref{portmanteau_eq11} and \eqref{portmanteau_eq12} together imply that for every non-negative bounded continuous function~$g$,
\begin{align}
\lim\limits_{n \to \infty} \sup_p \left| \int g \mathrm{d} \mu_{n, p} - \int g \mathrm{d} \mu_p \right| = 0 .
\label{portmanteau_eq13}
\end{align}
Now, let $f : \mathbb{R}^d \to \mathbb{R}$ be a bounded continuous function. Then, $f = f_+ + f_-$, where $f_+( \bx ) = f( \bx ) I( f( \bx ) \ge 0 )$ and $f_-( \bx ) = f( \bx ) I( f( \bx ) < 0 )$. Since $f$ is bounded and continuous, both $f_+$ and $f_-$ are bounded and continuous functions, and from \eqref{portmanteau_eq13} we have
\begin{align}
\begin{split}
& \lim\limits_{n \to \infty} \sup_p \left| \int f_+ \mathrm{d} \mu_{n, p} - \int f_+ \mathrm{d} \mu_p \right| = 0 ,\\
& \lim\limits_{n \to \infty} \sup_p \left| \int f_- \mathrm{d} \mu_{n, p} - \int f_- \mathrm{d} \mu_p \right| = 0 .
\end{split}
\label{portmanteau_eq14}
\end{align}
Now, for every $p$ and $n$,
\begin{align*}
\left| \int f \mathrm{d} \mu_{n, p} - \int f \mathrm{d} \mu_p \right|
& \le \left| \int f_+ \mathrm{d} \mu_{n, p} - \int f_+ \mathrm{d} \mu_p \right|
+ \left| \int f_- \mathrm{d} \mu_{n, p} - \int f_- \mathrm{d} \mu_p \right| \\
& \le \sup_p \left| \int f_+ \mathrm{d} \mu_{n, p} - \int f_+ \mathrm{d} \mu_p \right|
+ \sup_p \left| \int f_- \mathrm{d} \mu_{n, p} - \int f_- \mathrm{d} \mu_p \right| ,
\end{align*}
which implies, for every $n$,
\begin{align}
& \sup_p \left| \int f \mathrm{d} \mu_{n, p} - \int f \mathrm{d} \mu_p \right| \nonumber\\
& \le \sup_p \left| \int f_+ \mathrm{d} \mu_{n, p} - \int f_+ \mathrm{d} \mu_p \right|
+ \sup_p \left| \int f_- \mathrm{d} \mu_{n, p} - \int f_- \mathrm{d} \mu_p \right| .
\label{portmanteau_eq15}
\end{align}
\eqref{portmanteau_eq14} and \eqref{portmanteau_eq15} together imply that
\begin{align*}
\lim\limits_{n \to \infty} \sup_p \left| \int f \mathrm{d} \mu_{n, p} - \int f \mathrm{d} \mu_p \right| = 0 .
\end{align*}

We have established that \ref{p1} $\implies$ \ref{p2} $\implies$ \ref{p3} $\implies$ \ref{p4} $\implies$ \ref{p5} $\implies$ \ref{p1}. Next, we prove that \ref{p5} $\implies$ \ref{p6} $\implies$ \ref{p3}.

\ref{p5} $\implies$ \ref{p6}:
Let $\bx = (x_1, \ldots, x_d)^\top$ be a continuity point of $F_p(\cdot)$ for all $p$. Let $A = \{ \by = (y_1, \ldots, y_d)^\top \in \mathbb{R}^d : y_i \le x_i \text{ for } i = 1, \ldots, d \}$. Then $\partial A = \{ \by = (y_1, \ldots, y_d)^\top \in \mathbb{R}^d : y_i \le x_i \text{ for } i = 1, \ldots, d \text{ and } y_i = x_i \text{ for at least one } i \}$. Since $\mu_p(A) = F_p(\bx)$ and $\bx$ is a continuity point of $F_p(\cdot)$ for all $p$, $\mu_p(\partial A) = 0$ for all $p$, and hence $\sup_p | F_{n, p}( \bx ) - F_p( \bx ) | = \sup_p | \mu_{n, p}( A ) - \mu_p( A ) | \to 0$ as $n \to \infty$.

\ref{p6} $\implies$ \ref{p3}:
Define $A_{i, u} = \{ \bx = ( x_1, \ldots, x_d )^\top \in \mathbb{R}^d : x_i = u \}$ for $i = 1, \ldots, d$. Since for every $i$, the sets $A_{i, u}$ are disjoint for distinct values of $u$, there can be at most countably many values of $u$ such that $\mu_p( A_{i, u} ) > 0$ for at least one $p$. Hence, the set
\begin{align}
B = \{ u \in \mathbb{R} : \mu_p( A_{i, u} ) > 0 \text{ for at least one } p \text{ and at least one } i \}
\label{portmanteau_eq16}
\end{align}
is at most countable. We establish below that any $\bx = ( x_1, \ldots, x_d )^\top$ such that $x_i \notin B$ for every $i$ is a continuity point of $F_p( \cdot )$ for all $p$.

Suppose, $\bu = ( u_1, \ldots, u_d )^\top$ is not a continuity point of $F_p( \cdot )$. Then, there is $\epsilon_1 > 0$ and a sequence $\{ \bw_n^{(1)} = ( w_{n,1}^{(1)}, \ldots, w_{n,d}^{(1)} )^\top \}$ converging to $\bu$ such that $| F_p( \bu ) - F_p( \bw_n^{(1)} ) | > \epsilon_1$ for all $n$. Define the sets $B_i^{(1)} = \{ \bx = ( x_1, \ldots, x_d )^\top \in \mathbb{R}^d : x_i < u_i \}$ for $i = 1, \ldots, d$. Because $F_p( \cdot )$ is a distribution function, any sequence $\{ \bx_n = ( x_{n,1}, \ldots, x_{n,d} )^\top \}$ that converges to $\bu$ such that for all sufficiently large $n$, $x_{n,i} \ge u_i$ for all $i$, one must have $F_p( \bx_n ) \to F_p( \bu )$ as $n \to \infty$. Therefore, for all sufficiently large $n$, $\{ \bw_n^{(1)} \} \subset \cup_{i = 1}^d B_i^{(1)}$, and because there are only $d$ sets $B_i^{(1)}$, there is some $i_1$ such that $B_{i_1}^{(1)}$ contains infinitely many elements from $\{ \bw_n^{(1)} \}$, i.e., $\{ \bw_n^{(1)} \}$ has a subsequence $\{ \by_n^{(1)} =\{( y_{n,1}^{(1)}, \ldots, y_{n,d}^{(1)})^\top \}$ with $\{ \by_n^{(1)} \} \subset B_{i_1}^{(1)}$. Define a sequence $\{ \bz_n^{(1)} = ( z_{n,1}^{(1)}, \ldots, z_{n,d}^{(1)} )^\top \}$ such that $z_{n,i}^{(1)} = y_{n,i}^{(1)}$ for $i \neq i_1$ and $z_{n,i_1}^{(1)} = u_{i_1}$, and consider the sets $C_n^{(1)} = \{ \bx = ( x_1, \ldots, x_d )^\top \in \mathbb{R}^d : x_i \le y_{n,i}^{(1)} \text{ for } i \neq i_1 \text{ and } y_{n,i_1}^{(1)} < x_{i_1} \le u_{i_1} \}$ and $C^{(1)} = \{ \bx = ( x_1, \ldots, x_d )^\top \in \mathbb{R}^d : x_i \le y_{n,i}^{(1)} \text{ for } i \neq i_1 \text{ and } x_{i_1} = u_{i_1} \}$. Since $\{ \by_n^{(1)} \}$, being a subsequence of $\{ \bw_n^{(1)} \}$, converges to $\bu$, $y_{n,{i_1}}^{(1)} \to u_{i_1}$ as $n \to \infty$ as well as $y_{n,{i_1}}^{(1)} < u_{i_1}$ for all $n$, and this implies that the sequence of sets $\{ C_n^{(1)} \}$ converges to $C^{(1)}$. Therefore, $\mu_p( C_n^{(1)} ) \to \mu_p( C^{(1)} )$ as $n \to \infty$. Since $\mu_p( C_n^{(1)} ) = F_p( \bz_n^{(1)} ) - F_p( \by_n^{(1)} )$ for all $n$, if $F_p( \bz_n^{(1)} )$ converges to $F_p( \bu )$ as $n \to \infty$, we must have $\mu_p( C^{(1)} ) > \epsilon_1 / 2$, which implies $\mu_p( A_{i_1, u_{i_1}} ) > \epsilon_1 / 2$, because $C^{(1)} \subset A_{i_1, u_{i_1}}$.

On the other hand, if $F_p( \bz_n^{(1)} )$ does not converge to $F_p( \bu )$ as $n \to \infty$, there is $\epsilon_2 > 0$ and a subsequence $\{ \bw_n^{(2)} = ( w_{n,1}^{(2)}, \ldots, w_{n,d}^{(2)} )^\top \}$ of $\{ \bz_n^{(1)} \}$ such that $| F_p( \bu ) - F_p( \bw_n^{(2)} ) | > \epsilon_2$ for all $n$. Clearly, $\{ \bw_n^{(2)} \}$ converges to $\bu$ as $n \to \infty$ and $w_{n,i_1}^{(2)} = u_{i_1}$ for all $n$. Now, define the sets $B_i^{(2)} = \{ \bx = ( x_1, \ldots, x_d )^\top \in \mathbb{R}^d : x_i < u_i \text{ for } i \neq i_1 \text{ and } x_{i_1} = u_{i_1} \}$ for $i \in \{ 1, \ldots, d \}, \; i \neq i_1$. Using similar argument as before, we get a subsequence $\{ \by_n^{(2)} = ( y_{n,1}^{(2)}, \ldots, y_{n,d}^{(2)} )^\top \}$ of $\{ \bw_n^{(2)} \}$ with $\{ \by_n^{(2)} \} \subset B_{i_2}^{(2)}$ for some $i_2 \neq i_1$. Next, define the sequence $\{ \bz_n^{(2)} = ( z_{n,1}^{(2)}, \ldots, z_{n,d}^{(2)} )^\top \}$ such that $z_{n,i}^{(2)} = y_{n,i}^{(2)}$ for $i \neq i_2$ and $z_{n,i_2}^{(2)} = u_{i_2}$, and consider the sets $C_n^{(2)} = \{ \bx = ( x_1, \ldots, x_d )^\top \in \mathbb{R}^d : x_i \le y_{n,i}^{(2)} \text{ for } i \notin \{ i_1, i_2 \},\; y_{n,i_2}^{(2)} < x_{i_2} \le u_{i_2} \text{ and } x_{i_1} = u_{i_1} \}$ and $C^{(2)} = \{ \bx = ( x_1, \ldots, x_d )^\top \in \mathbb{R}^d : x_i \le y_{n,i}^{(2)} \text{ for } i \notin \{ i_1, i_2 \} \text{ and } x_i = u_i \text{ for } i \in \{ i_1, i_2 \} \}$. Again using arguments similar to those used before, we get $\mu_p( C_n^{(2)} ) \to \mu_p( C^{(2)} )$ as $n \to \infty$ and $\mu_p( C_n^{(2)} ) = F_p( \bz_n^{(2)} ) - F_p( \by_n^{(2)} )$ for all $n$. If $F_p( \bz_n^{(2)} )$ converges to $F_p( \bu )$ as $n \to \infty$, we have $\mu_p( C^{(2)} ) > \epsilon_2 / 2$, and this implies $\mu_p( A_{i_2, u_{i_2}} ) > \epsilon_2 / 2$, since $C^{(2)} \subset A_{i_2, u_{i_2}}$. If $F_p( \bz_n^{(2)} )$ does not converge to $F_p( \bu )$ as $n \to \infty$, there is $\epsilon_3 > 0$ and a subsequence $\{ \bw_n^{(3)} \}$ of $\{ \bz_n^{(2)} \}$ such that $| F_p( \bu ) - F_p( \bw_n^{(3)} ) | > \epsilon_3$ for all $n$, $\{ \bw_n^{(3)} \}$ converges to $\bu$ as $n \to \infty$ and $w_{n,i}^{(3)} = u_{i}$ for $i \in \{ i_1, i_2 \}$ and for all $n$, and we repeat the same argument.

Each iteration above fixes one additional component of the resultant sequence of $d$-dimensional vectors $\{ \bz_n^{(k)} \}$ to the corresponding component of $\bu$, and we continue the iterations if $F_p( \bz_n^{(k)} )$ does not converge to $F_p( \bu )$ as $n \to \infty$. If we still do not get convergence of $F_p( \bz_n^{(k)} )$ after $k = (d - 1)$ iterations, using the same arguments as in previous iterations, we get some $\epsilon_d > 0$ and a subsequence $\{ \by_n^{(d)} = ( y_{n,1}^{(d)}, \ldots, y_{n,d}^{(d)} )^\top \}$ of $\{ \bz_n^{(d - 1)} \}$ such that $y_{n,i}^{(d)} = u_i$ for $i \in \{ i_1, \ldots, i_{d-1} \}$, $y_{n,i_d}^{(d)} < u_{i_d}$ for all $n$, $\by_n^{(d)} \to \bu$ as $n \to \infty$ but $| F_p( \bu ) - F_p( \by_n^{(d)} ) | > \epsilon_d$ for all $n$. Now define $C_n^{(d)} = \{ \bx = ( x_1, \ldots, x_d )^\top \in \mathbb{R}^d : y_{n,i_2}^{(d)} < x_{i_d} \le u_{i_d} \text{ and } x_i = u_i \text{ for } i \in \{ i_1, \ldots, i_{d - 1} \} \}$. We have $\mu_p( C_n^{(d)} ) = F_p( \bu ) - F_p( \by_n^{(d)} ) = | F_p( \bu ) - F_p( \by_n^{(d)} ) | > \epsilon_d$ for all $n$. Also, the sets $C_n^{(d)} \to \{ \bu \}$ as $n \to \infty$ because $\by_n^{(d)}$ converges to $\bu$, which implies $\mu_p( \{ \bu \} ) \ge \epsilon_d$. But this implies that $\mu_p( A_{i, u_i} ) > 0$ for every $i$ since $\{ \bu \} = \cap_{i = 1}^d A_{i, u_i}$.
Hence, we have established that if $\bu = ( u_1, \ldots, u_d )^\top$ is not a continuity point of $F_p( \cdot )$, then $\mu_p( A_{i, u_i} ) > 0$ for at least one $i$.

Therefore, any $\bx = ( x_1, \ldots, x_d )^\top$ such that $x_i \notin B$ for every $i$ is a continuity point of $F_p( \cdot )$ for all $p$, where $B$ is described in \eqref{portmanteau_eq16}.

Next, let $\ba = ( a_1, \ldots, a_d )^\top$ and $\bb = ( b_1, \ldots, b_d )^\top$ be such that $a_i \notin B$, $b_i \notin B$ and $a_i < b_i$ for every $i$. Let $A = \{ \bx = ( x_1, \ldots, x_d )^\top \in \mathbb{R}^d : a_i < x_i \le b_i \text{ for } i = 1, \ldots, d \}$. We proceed to show that under \ref{p6}, $\sup_p | \mu_{n, p}( A ) - \mu_p( A ) | \to 0$ as $n \to \infty$. Let $\bc = ( c_1, \ldots, c_d )^\top$ be such that each $c_i$ is either $a_i$ or $b_i$, and let $n( \bc, \ba )$ denote the number of $a_i$'s in the components of the vector $\bc$. So, there are $2^d$ possible values of the vector $\bc$ depending on the values of $\ba$ and $\bb$. Let $I_{r, i}( \cdot ) : \mathbb{R}^d \to \mathbb{R}$ be defined as $I_{r, i}( \bx ) = I( x_i \le r )$, where $\bx = ( x_1, \ldots, x_d )^\top$ and $i = 1, \ldots, d$. Clearly, for any probability measure $\nu$ on $\mathbb{R}^d$, $\nu( A ) = \int \prod_{i = 1}^d \{ I_{b_i, i}( \bx ) - I_{a_i, i}( \bx ) \} \mathrm{d} \nu( \bx )$. Also, $\prod_{i = 1}^d \{ I_{b_i, i}( \bx ) - I_{a_i, i}( \bx ) \} = \sum_{\bc} (-1)^{ n( \bc, \ba ) } \prod_{i = 1}^d I_{c_i, i}( \bx )$, where the sum is over all $2^d$ values of $\bc$. Hence, $\nu( A ) = \sum_{\bc} (-1)^{ n( \bc, \ba ) } \int \prod_{i = 1}^d I_{c_i, i}( \bx ) \mathrm{d} \nu( \bx ) = \sum_{\bc} (-1)^{ n( \bc, \ba ) } \nu( B_\bc )$, where $B_\bc = \{ \bx = ( x_1, \ldots, x_d )^\top \in \mathbb{R}^d : x_i \le c_i \text{ for } i = 1, \ldots, d \}$. Since $\mu_{n, p}( B_\bc ) = F_{n, p}( \bc )$ and $\mu_p( B_\bc ) = F_p( \bc )$, we get that $\mu_{n, p}( A ) = \sum_{\bc} (-1)^{ n( \bc, \ba ) } F_{n, p}( \bc )$ and $\mu_p( A ) = \sum_{\bc} (-1)^{ n( \bc, \ba ) } F_p( \bc )$, which from \ref{p6} yields
\begin{align}
\sup_p | \mu_{n, p}( A ) - \mu_p( A ) | \le \sum_{\bc} (-1)^{ n( \bc, \ba ) } \sup_p \left| F_{n, p}( \bc ) - F_p( \bc ) \right| \to 0
\label{portmanteau_eq17}
\end{align}
as $n \to \infty$.

Let $f : \mathbb{R}^d \to \mathbb{R}$ be a continuous function which is zero outside the compact set $K$. While proving \ref{p2} $\implies$ \ref{p3} above, it was established that such a function $f$ is bounded and uniformly continuous. So, given $\epsilon > 0$, there is $\delta > 0$ such that $| \bx - \by | < \delta$ implies $| f( \bx ) - f( \by ) | < \epsilon$. We now construct a finite partition of $K$ in hyper-rectangles such that for every vertex $\bx = ( x_1, \ldots, x_d )^\top$ of any of the hyper-rectangles, $x_i \notin B$ for $i = 1, \ldots, d$, and hence every such vertex is a continuity point of $F_p( \cdot )$ for all $p$, and also, for points $\bx$ and $\by$ lying in a hyper-rectangle, $| \bx - \by | < \delta$.

Since $B$ given in \eqref{portmanteau_eq16} is at most countable, $B^c$ is dense in $\mathbb{R}$. Because $K$ is compact and $B^c$ is dense, there is $s > 0$ such that $s \in B^c$ and $K$ is a subset of the hyper-rectangle $( -s, s ]^d$. Take $s_1 = -s$, and given $s_k < s$ with $s - s_k \ge \delta / \sqrt{d}$, choose $s_{k + 1} > s_k$ such that $s_{k + 1} \in B^c$, $s_{k + 1} \le s$ and $\delta / 2 \sqrt{d} < s_{k + 1} - s_k < \delta / \sqrt{d}$. If $s - s_k < \delta / \sqrt{d}$, take $s_{k + 1} = s$. Such a sequence of $s_k$'s can be chosen because $B^c$ is dense. Also, by the method of construction, the sequence $\{ s_k \}$ is finite and can have at most $ ( 4 a \sqrt{d} / \delta ) + 1$ many elements. Let $q$ be the number of elements in $\{ s_k \}$. Then, the collection of intervals $\{ ( s_k, s_{k + 1} ] : k = 1, \ldots, q - 1 \}$ partitions $( -s, s ]$, and $s_i \notin B$ for $i = 1, \ldots, q$. Consider the corresponding partition of $( -s, s ]^d$ formed by hyper-rectangles with vertices $\bb = ( b_1, \ldots, b_d )^\top$ such that $b_i = s_j$ for some $j \in \{ 1, \ldots, q \}$ and for $i = 1, \ldots, d$. Clearly, this partition has at most $(q - 1)^d$ many hyper-rectangles, which are formed by taking products of the intervals $\{ ( s_k, s_{k + 1} ] \}$. Now, from this partition of $( -s, s ]^d$, we remove those hyper-rectangles $A$ such that $A \cap K = \emptyset$. Let $m$ hyper-rectangles remain in the collection, and we denote them as $A_1, \ldots, A_m$. The collection $\{ A_1, \ldots, A_m \}$ provides a  partition of $K$ in hyper-rectangles such that for every vertex $\bb = ( b_1, \ldots, b_d )^\top$ of any of the hyper-rectangles, $b_i \notin B$ for $i = 1, \ldots, d$. If $\bx = ( x_1, \ldots, x_d )^\top$ and $\by = ( y_1, \ldots, y_d )^\top$ lie in a hyper-rectangle, $| x_i - y_i | < \delta / \sqrt{d}$ for every $i$, and hence $| \bx - \by | < \delta$. Each $A_k$ is of the form $A_k = \prod_{i = 1}^d ( s_{k_i}, s_{k_i + 1} ]$, and define $\bz_k = ( z_{k, 1}, \ldots, z_{k, d} )^\top$ such that $z_{k, i} = ( s_{k_i} + s_{k_i + 1} ) / 2$ for $i = 1, \ldots, d$, i.e., $\bz_k$ is the mid-point of the hyper-rectangle $A_k$. In addition, from \eqref{portmanteau_eq17} we get~that
\begin{align}
\sup_p | \mu_{n, p}( A_k ) - \mu_p( A_k ) | \to 0 \text{ as } n \to \infty \text{ for } k = 1, \ldots, m .
\label{portmanteau_eq18}
\end{align}

Now, corresponding to $f$, we define another function $f_1 : \mathbb{R}^d \to \mathbb{R}$ in the following way. For any $\bx$, if $\bx \notin \cup_{i = 1}^m A_i$, then define $f_1( \bx ) = 0$. If $\bx \in \cup_{i = 1}^m A_i$, then there is exactly one hyper-rectangle $A_k$ which contains $\bx$, because $A_i$'s are disjoint, and define $f_1( \bx )$ for such an $\bx$ as $f_1( \bx ) = f( \bz_k )$, where $\bz_k$ is the mid-point of $A_k$. Note that for $\bx \notin \cup_{i = 1}^m A_i$, $f( \bx ) = f_1( \bx ) = 0$. And for $\bx \in A_i$ for any $A_i$, $| f( \bx ) - f_1( \bx ) | = | f( \bx ) - f( \bz_i ) | < \epsilon$, since $| \bx - \bz | < \delta$. Therefore,
\begin{align}
\sup_\bx | f( \bx ) - f_1( \bx ) |
& = \max\left\{ \sup_{\bx \in \cup_{i = 1}^m A_i} | f( \bx ) - f_1( \bx ) |, \sup_{\bx \notin \cup_{i = 1}^m A_i} | f( \bx ) - f_1( \bx ) | \right\} \nonumber\\
& = \sup\left\{ | f( \bx ) - f_1( \bx ) | : \bx \in \cup_{i = 1}^m A_i \right\} \nonumber\\
& = \max\left[ \sup\left\{ | f( \bx ) - f_1( \bx ) | : \bx \in A_i,\; i = 1, \ldots, m \right\} \right]
< \epsilon .
\label{portmanteau_eq19}
\end{align}
On the other hand, we have
\begin{align*}
\int f_1 \mathrm{d} \mu_{n, p} = \sum_{k = 1}^m f( \bz_k ) \mu_{n, p}( A_k ) \;\text{and}\;
\int f_1 \mathrm{d} \mu_p = \sum_{k = 1}^m f( \bz_k ) \mu_p( A_k ) ,
\end{align*}
which, from \eqref{portmanteau_eq18}, implies that
\begin{align}
\sup_p \left| \int f_1 \mathrm{d} \mu_{n, p} - \int f_1 \mathrm{d} \mu_p \right|
\le \sum_{k = 1}^m \left| f( \bz_k ) \right| \sup_p \left| \mu_{n, p}( A_k ) - \mu_p( A_k ) \right| \to 0
\label{portmanteau_eq20}
\end{align}
as $n \to \infty$.
Using \eqref{portmanteau_eq19} and \eqref{portmanteau_eq20}, we get
\begin{align*}
\sup_p \left| \int f \mathrm{d} \mu_{n, p} - \int f \mathrm{d} \mu_p \right|
& \le \sup_p \left| \int \{ f( \bx ) - f_1( \bx ) \} \mathrm{d} \mu_{n, p}( \bx ) \right| + \sup_p \left| \int \{ f( \bx ) - f_1( \bx ) \} \mathrm{d} \mu_p( \bx ) \right| \\
& \quad + \sup_p \left| \int f_1 \mathrm{d} \mu_{n, p} - \int f_1 \mathrm{d} \mu_p \right| \\
& \le 2 \sup_\bx\left| f( \bx ) - f_1( \bx ) \right| 
+ \sup_p \left| \int f_1 \mathrm{d} \mu_{n, p} - \int f_1 \mathrm{d} \mu_p \right| \\
& < 3 \epsilon
\end{align*}
for all sufficiently large $n$. This establishes \ref{p6} $\implies$ \ref{p3} since $\epsilon > 0$ is arbitrary.
\end{proof}

We need the following result (Lemma 8.7.1 in \cite{lebanon2012probability}) for proving the uniform-over-$p$ L\'{e}vy's continuity theorem.
\begin{lemma}[\cite{lebanon2012probability}]\label{lem3}
For $ a > 0 $ and $b \in \mathbb{R}$,
\begin{align*}
\int_{-\infty}^\infty \exp\left( - a x^2 + b x \right) \mathrm{d} x = \sqrt{\frac{\pi}{a}} \exp\left( \frac{b^2}{4 a} \right) .
\end{align*}
\end{lemma}
%\begin{proof}
%\begin{align*}
%& \int \exp\left( - a x^2 + b x \right) \mathrm{d} x \\
%& = \exp\left( \frac{b^2}{4 a} \right) \int \exp\left( - \frac{\left( x - \frac{b}{2 a} \right)^2}{2 (1 / 2 a)} \right) \mathrm{d} x
%= \exp\left( \frac{b^2}{4 a} \right) \sqrt{\frac{\pi}{a}} .
%\end{align*}
%\end{proof}

\begin{proof}[Proof of \autoref{levy}]
%The proof follows arguments analogous to those in the proof of Proposition 8.8.1 in \cite{lebanon2012probability}.

Let $ \bX_{n, p} $ and $ \bX_p $ be random vectors with distributions $ \mu_{n, p} $ and $ \mu_p $, respectively.
Suppose $ \mu_{n, p} \Longrightarrow \mu_p $ uniformly over $ p $. For any fixed $ \bt \in \mathbb{R}^d $, $ \cos( \bt^\top \bx ) $ and $ \sin( \bt^\top \bx ) $ are bounded continuous functions of $ \bx $. Hence, from \autoref{definition1}, we have
\begin{align*}
\lim\limits_{n \to \infty}\sup_p \left| \E\left[ \cos\left( \bt^\top \bX_{n, p} \right) \right] - \E\left[ \cos\left( \bt^\top \bX_p \right) \right] \right| = 0 ,\\
\lim\limits_{n \to \infty}\sup_p \left| \E\left[ \sin\left( \bt^\top \bX_{n, p} \right) \right] - \E\left[ \sin\left( \bt^\top \bX_p \right) \right] \right| = 0 .
\end{align*}
Therefore,
\begin{align*}
& \sup_p | \varphi_{n, p}( \bt ) - \varphi_p( \bt ) | \\
& = \sup_p \left( \left| \E\left[ \cos\left( \bt^\top \bX_{n, p} \right) \right] - \E\left[ \cos\left( \bt^\top \bX_p \right) \right] \right|^2 
+ \left| \E\left[ \sin\left( \bt^\top \bX_{n, p} \right) \right] - \E\left[ \sin\left( \bt^\top \bX_p \right) \right] \right|^2 \right)^{\frac{1}{2}} \\
& \le \sup_p \left\{ \left| \E\left[ \cos\left( \bt^\top \bX_{n, p} \right) \right] - \E\left[ \cos\left( \bt^\top \bX_p \right) \right] \right| 
+ \left| \E\left[ \sin\left( \bt^\top \bX_{n, p} \right) \right] - \E\left[ \sin\left( \bt^\top \bX_p \right) \right] \right| \right\} \\
& \le \sup_p \left| \E\left[ \cos\left( \bt^\top \bX_{n, p} \right) \right] - \E\left[ \cos\left( \bt^\top \bX_p \right) \right] \right| 
+ \sup_p \left| \E\left[ \sin\left( \bt^\top \bX_{n, p} \right) \right] - \E\left[ \sin\left( \bt \bX_p \right) \right] \right| \\
& \to 0 \text{ as } n \to \infty .
\end{align*}

Next, assume $ \sup_p | \varphi_{n, p}( \bt ) - \varphi( \bt ) | \to 0 $ as $ n \to \infty $. We shall show that for any continuous function $ g : \mathbb{R}^d \to \mathbb{R} $ that is zero outside a compact set, we have $ \sup_p \left| \E\left[ g\left( \bX_{n, p} \right) \right] - \E\left[ g\left( \bX_p \right) \right] \right| \to 0 $ as $ n \to \infty $. Define $ g_+( \bx ) = g( \bx ) I( g( \bx ) \ge 0 ) $ and $ g_-( \bx ) = | g( \bx ) | I( g( \bx ) < 0 ) $. Note that
\begin{align*}
\left| \E\left[ g\left( \bX_{n, p} \right) \right] - \E\left[ g\left( \bX_p \right) \right] \right|
& \le \left| \E\left[ g_+\left( \bX_{n, p} \right) \right] - \E\left[ g_+\left( \bX_p \right) \right] \right|
+ \left| \E\left[ g_-\left( \bX_{n, p} \right) \right] - \E\left[ g_-\left( \bX_p \right) \right] \right| ,
\end{align*}
which implies
\begin{align*}
\sup_p \left| \E\left[ g\left( \bX_{n, p} \right) \right] - \E\left[ g\left( \bX_p \right) \right] \right|
& \le \sup_p \left| \E\left[ g_+\left( \bX_{n, p} \right) \right] - \E\left[ g_+\left( \bX_p \right) \right] \right|
+ \sup_p \left| \E\left[ g_-\left( \bX_{n, p} \right) \right] - \E\left[ g_-\left( \bX_p \right) \right] \right| .
\end{align*}
Since $ g_+( \bx ) $ and $ g_-( \bx ) $ are non-negative functions of $ \bx $, it is enough to consider $ g $ as a non-negative continuous function which is zero outside a compact set. If $ g $ is identically 0, the assertion $ \sup_p \left| \E\left[ g\left( \bX_{n, p} \right) \right] - \E\left[ g\left( \bX_p \right) \right] \right| \to 0 $ as $ n \to \infty $ is vacuously satisfied. So, without loss of generality, we assume that $ g $ is a non-negative and non-zero continuous function which is zero outside a compact set.

Let $ \epsilon > 0 $ be fixed.
As $ g $ is continuous on a compact set, it is uniformly continuous, so we can find $ \delta > 0 $ such that $ | \bx - \by | \le \delta $ implies $ | g( \bx ) - g( \by ) | < \epsilon $.

Let $ \bZ $ be a $d$-dimensional random vector independent of $ \bX_{n, p} $ and $ \bX_p $ for all $ n $ and $ p $ and follow the $ {\cal N}_d( \mathbf{0}, \sigma^2 \mathbf{I}_d ) $ distribution, where $ \sigma^2 $ is small such that $ P[ \| \bZ \| > \delta ] < \epsilon $. We have
\begin{align}
\left| \E\left[ g\left( \bX_{n, p} \right) \right] - \E\left[ g\left( \bX_p \right) \right] \right|
 \le & \left| \E\left[ g\left( \bX_{n, p} \right) \right] - \E\left[ g\left( \bX_{n, p} + \bZ \right) \right] \right| \nonumber\\
& + \left| \E\left[ g\left( \bX_{n, p} + \bZ \right) \right] - \E\left[ g\left( \bX_p + \bZ \right) \right] \right| \nonumber\\
& + \left| \E\left[ g\left( \bX_p + \bZ \right) \right] - \E\left[ g\left( \bX_p \right) \right] \right| . \label{eq:levy1}
\end{align}
Now,
\begin{align}
\sup_p \left| \E\left[ g\left( \bX_{n, p} \right) \right] - \E\left[ g\left( \bX_{n, p} + \bZ \right) \right] \right| 
& \le \sup_p \left| \E\left[ \left\{ g\left( \bX_{n, p} \right) - g\left( \bX_{n, p} + \bZ \right) \right\} \left\{ I( \| \bZ \| \le \delta ) + I( \| \bZ \| > \delta ) \right\} \right] \right| \nonumber\\
& \le \sup_p \E\left[ \left| g\left( \bX_{n, p} \right) - g\left( \bX_{n, p} + \bZ \right) \right| I( \| \bZ \| \le \delta ) \right] \nonumber\\
& \quad + \sup_p \E\left[ \left| g\left( \bX_{n, p} \right) - g\left( \bX_{n, p} + \bZ \right) \right| I( \| \bZ \| > \delta ) \right] \nonumber\\
& < \epsilon + 2 \sup_\bx | g( \bx ) | P[ \| \bZ \| > \delta ]
= \epsilon \left( 1 + 2 \sup_\bx | g( \bx ) | \right) . \label{eq:levy2}
\end{align}
Similarly,
\begin{align}
\sup_p \left| \E\left[ g\left( \bX_p \right) \right] - \E\left[ g\left( \bX_p + \bZ \right) \right] \right|
< \epsilon \left( 1 + 2 \sup_\bx | g( \bx ) | \right) . \label{eq:levy3}
\end{align}
Next, using \autoref{lem3} and Fubini's theorem, we have
\begin{align}
& \E\left[ g\left( \bX_{n, p} + \bZ \right) \right] \nonumber\\
& = \frac{1}{\left( \sqrt{2 \pi} \sigma \right)^d} \int \int g( \bx + \bz ) \exp\left( - \frac{\bz^\top \bz}{2 \sigma^2} \right) \mathrm{d} \bz\, \mathrm{d} \mu_{n, p}( \bx ) \nonumber\\
& = \frac{1}{\left( \sqrt{2 \pi} \sigma \right)^d} \int \int g( \bu ) \exp\left( - \frac{( \bu - \bx )^\top ( \bu - \bx )}{2 \sigma^2} \right) \mathrm{d} \bu\, \mathrm{d} \mu_{n, p}( \bx ) \nonumber\\
& = \frac{1}{\left( \sqrt{2 \pi} \sigma \right)^d} \int \int g( \bu ) \prod_{i = 1}^d \exp\left( - \frac{( u_i - x_i )^2}{2 \sigma^2} \right) \mathrm{d} \bu\, \mathrm{d} \mu_{n, p}( \bx ) \nonumber\\
& = \frac{1}{\left( \sqrt{2 \pi} \sigma \right)^d} \int \int g( \bu ) \prod_{i = 1}^d \frac{\sigma}{\sqrt{2 \pi}} \int \exp\left( \mathrm{i} t_i ( u_i - x_i ) - \frac{\sigma^2 t_i^2}{2} \right) \mathrm{d} t_i\, \mathrm{d} \bu\, \mathrm{d} \mu_{n, p}( \bx ) \nonumber\\
& = \frac{1}{( 2 \pi )^d} \int \int \int g( \bu ) \exp\left( \mathrm{i} \bt^\top \bu - \frac{\sigma^2 \bt^\top \bt}{2} \right) e^{- \mathrm{i} \bt^\top \bx} \mathrm{d} \mu_{n, p}( \bx ) \, \mathrm{d} \bt \, \mathrm{d} \bu \nonumber\\
& = \frac{1}{( 2 \pi )^d} \int \int g( \bu ) \exp\left( \mathrm{i} \bt^\top \bu - \frac{\sigma^2 \bt^\top \bt}{2} \right) \varphi_{n, p}( - \bt ) \mathrm{d} \bt \, \mathrm{d} \bu \nonumber\\
& = \frac{\int g( \bu ) \mathrm{d} \bu}{( \sqrt{2 \pi} \sigma )^d} \int \int \exp\left( \mathrm{i} \bt^\top \bu \right) \varphi_{n, p}( - \bt ) \frac{g( \bu )}{\int g( \bu ) \mathrm{d} \bu} \left( \frac{\sigma}{\sqrt{2 \pi}} \right)^d \exp\left( - \frac{\sigma^2 \bt^\top \bt}{2} \right) \mathrm{d} \bt \, \mathrm{d} \bu \nonumber\\
& = \frac{\int g( \bu ) \mathrm{d} \bu}{( \sqrt{2 \pi} \sigma )^d} \E\left[ \exp\left( \mathrm{i} \bT^\top \bU \right) \varphi_{n, p}( - \bT) \right] , \label{eq:levy4}
\end{align}
where $ \bT $ and $ \bU $ are independent random vectors, with $ \bT $ following the distribution $ {\cal N}_d( {\bf 0}, \sigma^{-2} \mathbf{I}_d ) $ and $ \bU $ has the density $ \left( \int g( \bu ) \mathrm{d} \bu \right)^{-1} g( \bu ) $.

Similarly,
\begin{align}
\E\left[ g\left( \bX_p + \bZ \right) \right]
& =  \frac{\int g( \bu ) \mathrm{d} \bu}{( \sqrt{2 \pi} \sigma )^d} \E\left[ \exp\left( \mathrm{i} \bT^\top \bU \right) \varphi_p( - \bT) \right]. \label{eq:levy5}
\end{align}
From \eqref{eq:levy4} and \eqref{eq:levy5} we have
\begin{align}
& \left| \E\left[ g\left( \bX_{n, p} + \bZ \right) \right] - \E\left[ g\left( \bX_p + \bZ \right) \right] \right| \nonumber\\
& = \frac{\int g( \bu ) \mathrm{d} \bu}{( \sqrt{2 \pi} \sigma )^d} \left| \E\left[ \exp\left( \mathrm{i} \bT^\top \bU \right) \varphi_{n, p}( - \bT ) \right] - \E\left[ \exp\left( \mathrm{i} \bT^\top \bU \right) \varphi_p( - \bT ) \right] \right| \nonumber\\
& \le \frac{\int g( \bu ) \mathrm{d} \bu}{( \sqrt{2 \pi} \sigma )^d} \E\left[ \left| \exp\left( \mathrm{i} \bT^\top \bU \right) \varphi_{n, p}( - \bT ) - \exp\left( \mathrm{i} \bT^\top \bU \right) \varphi_p( - \bT ) \right| \right] \nonumber\\
& = \frac{\int g( \bu ) \mathrm{d} \bu}{( \sqrt{2 \pi} \sigma )^d} \E\left[ \left| \exp\left( \mathrm{i} \bT^\top \bU \right) \right| \left| \varphi_{n, p}( - \bT ) - \varphi_p( - \bT ) \right| \right] \nonumber\\
& = \frac{\int g( \bu ) \mathrm{d} \bu}{( \sqrt{2 \pi} \sigma )^d} \E\left[ \left| \varphi_{n, p}( - \bT ) - \varphi_p( - \bT ) \right| \right] . \label{eq:levy6}
\end{align}
Since $ \sup_p | \varphi_{n, p}( \bt ) - \varphi_p( \bt ) | \to 0 $ as $ n \to \infty $ and $ \sup_p | \varphi_{n, p}( \bt ) - \varphi_p( \bt ) | $ is a measurable function of $ \bt $, we have from \eqref{eq:levy6},
\begin{align*}
\sup_p \left| \E\left[ g\left( \bX_{n, p} + \bZ \right) \right] - \E\left[ g\left( \bX_p + \bZ \right) \right] \right|
& \le \frac{\int g( \bu ) \mathrm{d} \bu}{( \sqrt{2 \pi} \sigma )^d} \sup_p \E\left[ \left| \varphi_{n, p}( - \bT ) - \varphi_p( - \bT ) \right| \right] \\
& \le \frac{\int g( \bu ) \mathrm{d} u}{( \sqrt{2 \pi} \sigma )^d} E\left[ \sup_p \left| \varphi_{n, p}( - \bT ) - \varphi_p( - \bT ) \right| \right] .
\end{align*}
By the bounded convergence theorem, $ \E\left[ \sup_p \left| \varphi_{n, p}( - \bT ) - \varphi( - \bT ) \right| \right] \to 0 $ as $ n \to \infty $, and hence for all sufficiently large $ n $,
\begin{align}
\sup_p \left| \E\left[ g\left( \bX_{n, p} + \bZ \right) \right] - \E\left[ g\left( \bX_p + \bZ \right) \right] \right| < \epsilon .\label{eq:levy7}
\end{align} 
From \eqref{eq:levy1}, \eqref{eq:levy2}, \eqref{eq:levy3} and \eqref{eq:levy7}, we have
\begin{align*}
\sup_p \left| \E\left[ g\left( \bX_{n, p} \right) \right] - \E\left[ g\left( \bX_p \right) \right] \right|
\le \epsilon \left( 3 + 4 \sup_\bx | g( \bx ) | \right) .
\end{align*}
Since $ \epsilon > 0 $ is arbitrary, we have
\begin{align*}
\lim\limits_{n \to \infty} \sup_p \left| \E\left[ g\left( \bX_{n, p} \right) \right] - \E\left[ g\left( \bX_p \right) \right] \right| = 0 .
\end{align*}
\end{proof}

% Proof of \autoref{mappingthm}, \autoref{lemmaSlutsky} and \autoref{Slutsky} are placed in the Supplement.
\begin{proof}[Proof of \autoref{mappingthm}]
For any bounded and continuous $f : \mathbb{R} \to \mathbb{R}$, the composite function $h : \mathbb{R}^d \to \mathbb{R}$ defined by $h( \bx ) = f( g( \bx ) )$ is bounded and continuous. Since $\bX_{n, p} \Longrightarrow \bX_p$ uniformly-over-$p$, from \autoref{definition1} we have
\begin{align*}
\sup_p \left| \E[ f( g( \bX_{n, p} ) ) ] - \E[ f( g( \bX_p ) ) ] \right|
= \sup_p \left| \E[ h( \bX_{n, p} ) ] - \E[ h( \bX_p ) ] \right| \to 0
\end{align*}
as $n \to \infty$, which implies that $g\left( \bX_{n, p} \right) \Longrightarrow g\left( \bX_p \right)$ uniformly-over-$p$.
\end{proof}

\begin{proof}[Proof of \autoref{lemmaSlutsky}]
Let $f : \mathbb{R}^2 \to \mathbb{R}$ be a bounded and uniformly continuous function. So, given any $\epsilon > 0$, there is $\delta > 0$ such that $\| ( x_1, x_2 ) - ( y_1, y_2 ) \| \le \delta$ implies that $| f( x_1, x_2 ) - f( y_1, y_2 ) | < \epsilon$.
Therefore, $| c_1 - c_2 | \le \delta$ implies that for any probability measure $\nu$ on $\mathbb{R}^d$,
\begin{align}
& \left| \int f( x, c_1 ) \mathrm{d} \nu( x ) - \int f( x, c_2 ) \mathrm{d} \nu( x ) \right|
\le \int \left| f( x, c_1 ) \mathrm{d} \nu( x ) - f( x, c_2 ) \right| \mathrm{d} \nu( x )
< \epsilon .
\label{lemmaSlutsky_eq1}
\end{align}
Since $\{ c_p \}$ is a bounded sequence, we can construct a finite collection $\{ a_1, \ldots, a_m \} \subset \mathbb{R}$ such that for every $c_p$, $| c_p - a_k | < \delta$ for at least one $a_k$. Let $b_p$ denote the number $a_k \in \{ a_1, \ldots, a_m \}$ closest to $c_p$. Therefore, using \eqref{lemmaSlutsky_eq1}, we get
\begin{align*}
\left| \E[ f( X_{n, p}, c_p ) ] - \E[ f( X_{n, p}, b_p ) ] \right|
+ \left| \E[ f( X_p, c_p ) ] - \E[ f( X_p, b_p ) ] \right|
< 2 \epsilon ,
\end{align*}
which implies that for every $p$ and every $n$,
\begin{align}
& \min\{ \left| \E[ f( X_{n, p}, c_p ) ] - \E[ f( X_{n, p}, a_k ) ] \right| \nonumber\\
& \qquad
+ \left| \E[ f( X_p, c_p ) ] - \E[ f( X_p, a_k ) ] \right| : k = 1, \ldots, m \} \nonumber\\
& < 2 \epsilon .
\label{lemmaSlutsky_eq2}
\end{align}
Next, define $g_c : \mathbb{R} \to \mathbb{R}$ by $g_c( x ) = f( x, c )$ for $c \in \mathbb{R}$. Then, $g_c$ is a bounded and uniformly continuous function for every $c$, and from \ref{p2} in \autoref{portmanteau} we have
\begin{align*}
\sup_p \left| \E[ f( X_{n, p}, c ) ] - \E[ f( X_p, c ) ] \right| =
\sup_p \left| \E[ g_c( X_{n, p} ) ] - \E[ g_c( X_p ) ] \right| \to 0
\end{align*}
as $n \to \infty$. Therefore, there is $n_1$ such that for all $n \ge n_1$,
\begin{align}
\max\left\{ \sup_p \left| \E[ f( X_{n, p}, a_k ) ] - \E[ f( X_p, a_k ) ] \right| : k = 1, \ldots, m \right\}
< \epsilon .
\label{lemmaSlutsky_eq3}
\end{align}
On the other hand, we have for every $p$,
\begin{align*}
& \left| \E[ f( X_{n, p}, c_p ) ] - \E[ f( X_p, c_p ) ] \right| \\
& \le \min\{ \left| \E[ f( X_{n, p}, c_p ) ] - \E[ f( X_{n, p}, a_k ) ] \right| 
+ \left| \E[ f( X_p, c_p ) ] - \E[ f( X_p, a_k ) ] \right| \\
& \qquad\quad\;
+ \left| \E[ f( X_{n, p}, a_k ) ] - \E[ f( X_p, a_k ) ] \right|
: k = 1, \ldots, m \} \\
& \le \min\{ \left| \E[ f( X_{n, p}, c_p ) ] - \E[ f( X_{n, p}, a_k ) ] \right| 
+ \left| \E[ f( X_p, c_p ) ] - \E[ f( X_p, a_k ) ] \right| : k = 1, \ldots, m \} \\
& \quad
+ \max\left\{ \sup_p \left| \E[ f( X_{n, p}, a_k ) ] - \E[ f( X_p, a_k ) ] \right| : k = 1, \ldots, m \right\} ,
\end{align*}
which, from an application of \eqref{lemmaSlutsky_eq2} and \eqref{lemmaSlutsky_eq3} together, yields
\begin{align*}
& \left| \E[ f( X_{n, p}, c_p ) ] - \E[ f( X_p, c_p ) ] \right| < 3 \epsilon
\end{align*}
for every $n \ge n_1$ and every $p$. Hence, for all $n \ge n_1$,
\begin{align}
\sup_p \left| \E[ f( X_{n, p}, c_p ) ] - \E[ f( X_p, c_p ) ] \right| < 3 \epsilon .
\label{lemmaSlutsky_eq4}
\end{align}
Let $l = \sup\{ | f( x, y ) | : ( x, y ) \in \mathbb{R}^2 \}$.
Now, $\| ( X_{n, p}, Y_{n, p} ) - ( X_{n, p}, c_p ) \| = | Y_{n, p} - c_p |$. So, $| Y_{n, p} - c_p | \le \delta$ implies $| f( X_{n, p}, Y_{n, p} ) - f( X_{n, p}, c_p ) | < \epsilon$, and using this fact we have
\begin{align*}
& | \E[ f( X_{n, p}, Y_{n, p} ) ] - \E[ f( X_{n, p}, c_p ) ] | \\
& \le \left| \E[ \{ f( X_{n, p}, Y_{n, p} ) - f( X_{n, p}, c_p ) \} I( | Y_{n, p} - c_p | \le \delta ) ] \right| \\
& \quad + \left| \E[ \{ f( X_{n, p}, Y_{n, p} ) - f( X_{n, p}, c_p ) \} I( | Y_{n, p} - c_p | > \delta ) ] \right| \\
& \le \E\left[ | f( X_{n, p}, Y_{n, p} ) - f( X_{n, p}, c_p ) | I( | Y_{n, p} - c_p | \le \delta ) \right] \\
& \quad + 2 \sup\{ | f( x, y ) | : ( x, y ) \in \mathbb{R}^2 \} \P[ | Y_{n, p} - c_p | > \delta ] \\
& \le \epsilon \P\left[ | Y_{n, p} - c_p | \le \delta \right]
+ 2 l \P[ | Y_{n, p} - c_p | > \delta ] ,
\end{align*}
which implies
\begin{align}
& \sup_p \left| \E[ f( X_{n, p}, Y_{n, p} ) ] - \E[ f( X_{n, p}, c_p ) ] \right| \nonumber\\
& \le \epsilon \sup_p \P\left[ | Y_{n, p} - c_p | \le \delta \right]
+ 2 l \sup_p \P[ | Y_{n, p} - c_p | > \delta ] .
\label{lemmaSlutsky_eq5}
\end{align}
Since $Y_{n, p} \stackrel{\P}{\longrightarrow} c_p$ uniformly-over-$p$, from \autoref{definition2} we get that there is $n_2$ such that for all $n \ge n_2$, $\sup_p \P[ | Y_{n, p} - c_p | > \delta ] < \epsilon$. Hence, from \eqref{lemmaSlutsky_eq5} we have
\begin{align}
\sup_p \left| \E[ f( X_{n, p}, Y_{n, p} ) ] - \E[ f( X_{n, p}, c_p ) ] \right|
\le \epsilon ( 1 + 2 l )
\label{lemmaSlutsky_eq6}
\end{align}
for all $n \ge n_2$. Using \eqref{lemmaSlutsky_eq4} and \eqref{lemmaSlutsky_eq6}, we get that for all $n \ge \max\{ n_1, n_2 \}$,
\begin{align}
& \sup_p \left| \E[ f( X_{n, p}, Y_{n, p} ) ] - \E[ f( X_p, c_p ) ] \right| \nonumber\\
& \le \sup_p \left| \E[ f( X_{n, p}, Y_{n, p} ) ] - \E[ f( X_{n, p}, c_p ) ] \right| 
+ \sup_p \left| \E[ f( X_{n, p}, c_p ) ] - \E[ f( X_p, c_p ) ] \right| \nonumber\\
& < \epsilon ( 1 + 2 l ) + 3 \epsilon = 2 \epsilon ( 2 + l ) .
\label{lemmaSlutsky_eq7}
\end{align}
Since $\epsilon > 0$ is arbitrary, \eqref{lemmaSlutsky_eq7} along with \ref{p2} in \autoref{portmanteau} imply that $\left( X_{n, p}, Y_{n, p} \right) \Longrightarrow \left( X_p, c_p \right)$ uniformly-over-$p$.
\end{proof}

\begin{proof}[Proof of \autoref{Slutsky}]
From \autoref{lemmaSlutsky}, we get that $\left( X_{n, p}, Y_{n, p} \right) \Longrightarrow \left( X_p, c_p \right)$ uniformly-over-$p$.
Now, to prove \ref{s1}, define $g : \mathbb{R}^2 \to \mathbb{R}$ by $g( x, y ) = x + y$. Then, from the continuity of $g$ and \autoref{mappingthm}, we get $X_{n, p} + Y_{n, p} \Longrightarrow X_p + c_p$ uniformly-over-$p$.
To prove \ref{s2}, define $g : \mathbb{R}^2 \to \mathbb{R}$ by $g( x, y ) = x y$, and we similarly get $ X_{n, p} Y_{n, p} \Longrightarrow c_p X_p$ uniformly-over-$p$.

Next, let $l = \inf_p | c_p |$ and define $g : \mathbb{R}^2 \to \mathbb{R}$ by $g( x, y ) = x / y$ if $| y | \ge l / 2$ and $g( x, y ) = 2 x / l$ if $| y | < l / 2$. Clearly, $g$ is a continuous function on $\mathbb{R}^2$. Also, $g( x, y ) \neq x / y$ if and only if $| y | < l / 2$, and in particular $g( x, c_p ) = x / c_p$ for any $x$ and every $p$, since $| c_p | \ge l$ for every $p$. Hence, for any bounded continuous function $f : \mathbb{R} \to \mathbb{R}$,
\begin{align*}
& | \E[ f( X_{n, p} / Y_{n, p} ) ] - \E[ f( X_p / c_p ) ] | \\
& \le \E[ | f( X_{n, p} / Y_{n, p} ) - f( g( X_{n, p}, Y_{n, p} ) ) | ] \\
& \quad + \E[ | f( g( X_p, c_p ) ) - f( X_p / c_p ) | ] \\
& \quad + | \E[ f( g( X_{n, p}, Y_{n, p} ) ) ] - \E[ f( g( X_p, c_p ) ) ] | \\
& = \E[ | f( X_{n, p} / Y_{n, p} ) - f( g( X_{n, p}, Y_{n, p} ) ) | I( | Y_{n, p} | < l / 2 ) ] \\
& \quad + | \E[ f( g( X_{n, p}, Y_{n, p} ) ) ] - \E[ f( g( X_p, c_p ) ) ] | \\
& \le 2 \sup\{ | f( x, y ) | : ( x, y ) \in \mathbb{R}^2 \} \P[ | Y_{n, p} | < l / 2 ] \\
& \quad + | \E[ f( g( X_{n, p}, Y_{n, p} ) ) ] - \E[ f( g( X_p, c_p ) ) ] | ,
\end{align*}
which implies
\begin{align}
& \sup_p | \E[ f( X_{n, p} / Y_{n, p} ) ] - \E[ f( X_p / c_p ) ] | \nonumber\\
& \le 2 \sup\{ | f( x, y ) | : ( x, y ) \in \mathbb{R}^2 \} \sup_p \P[ | Y_{n, p} | < l / 2 ] \nonumber\\
& \quad + \sup_p | \E[ f( g( X_{n, p}, Y_{n, p} ) ) ] - \E[ f( g( X_p, c_p ) ) ] | .
\label{Slutsky_eq1}
\end{align}
Now, $| Y_{n, p} | < l / 2$ implies that $| Y_{n, p} - c_p | \ge | c_p | - | Y_{n, p} | > l - ( l / 2 ) = l / 2$. Hence, $\P[ | Y_{n, p} | < l / 2 ] \le \P[ | Y_{n, p} - c_p | > l / 2 ]$ for every $n$ and $p$. So, from \eqref{Slutsky_eq1} we have
\begin{align}
& \sup_p | \E[ f( X_{n, p} / Y_{n, p} ) ] - \E[ f( X_p / c_p ) ] | \nonumber\\
& \le 2 \sup\{ | f( x, y ) | : ( x, y ) \in \mathbb{R}^2 \} \sup_p \P[ | Y_{n, p} - c_p | > l / 2 ] \nonumber\\
& \quad + \sup_p | \E[ f( g( X_{n, p}, Y_{n, p} ) ) ] - \E[ f( g( X_p, c_p ) ) ] | .
\label{Slutsky_eq2}
\end{align}
Consider any $\epsilon > 0$.
Since from \autoref{lemmaSlutsky} we have $( X_{n, p}, Y_{n, p} ) \Longrightarrow ( X_p, c_p )$ uniformly-over-$p$ and the composite function $h : \mathbb{R}^2 \to \mathbb{R}$ defined by $h( x, y ) = f( g( x, y ) )$ is bounded and continuous, from \autoref{definition1} we get that there is an integer $n_1$ such that such for all $n \ge n_1$, we have $\sup_p | \E[ f( g( X_{n, p}, Y_{n, p} ) ) ] - \E[ f( g( X_p, c_p ) ) ] | < \epsilon$. Also, from \autoref{definition2} we get that there is another integer $n_2$ such that for all $n \ge n_2$, we have $\sup_p \P[ | Y_{n, p} - c_p | > l / 2 ] < \epsilon$. Therefore, from \eqref{Slutsky_eq2} we get that for all $n \ge \max\{ n_1, n_2 \}$,
\begin{align*}
& \sup_p | \E[ f( X_{n, p} / Y_{n, p} ) ] - \E[ f( X_p / c_p ) ] | < \epsilon ( 1 + 2 \sup\{ | f( x, y ) | : ( x, y ) \in \mathbb{R}^2 \} ) .
\end{align*}
Since $\epsilon > 0$ is arbitrary and $f$ is any bounded continuous function, the above inequality establishes \ref{s3}.
\end{proof}

\begin{proof}[Proof of \autoref{lemma_gaussian}]
Let $\{ \be_{p, i} \}$ be the eigenvectors of $\bSigma_p$ corresponding to the eigenvalues $\{ \delta_{p, i} \}$.
By the Karhunen-Loeve expansion, we have
\begin{align*}
\bZ_{p, i} = \sum_{k = 1}^p \sqrt{\delta_{p, k}} V_{p, k, i} \be_{p, k}
\end{align*}
for $i = 1, \ldots, n$, where $\{ V_{p, k, i} \}$ are random variables uncorrelated over $k = 1, \ldots, p$ and independent over $i = 1, \ldots, n$, with $\E[ V_{p, k, i} ] = 0$ and $\Var( V_{p, k, i} ) = 1$ for all $k$ and $i$. Denote
$\bar{V}_{p, k} = n^{-1} \sum_{i=1}^n V_{p, k, i}$. Then, $\bar{\bZ}_{p} = \sum_{k = 1}^p \sqrt{\delta_{p, k}} \bar{V}_{p, k} \be_{p, k}$, and
\begin{align*}
n \| \bar{\bZ}_{p} \|^2 = \sum_{k = 1}^p \delta_{p, k} n ( \bar{V}_{p, k} )^2 .
\end{align*}
Define
\begin{align*}
U_{n, p} = \frac{n \| \bar{\bZ}_{p} \|^2 - \sum_{k = 1}^p \delta_{p, k}}{\left( 2 \sum_{k = 1}^p \delta_{p, k}^2 \right)^{1/2}}
\text{ and }
\tilde{U}_{n, q} = \frac{\sum_{k = 1}^q \delta_{p, k} \left\{ ( \bar{V}_{p, k} )^2 - 1 \right\}}{\left( 2 \sum_{k = 1}^p \delta_{p, k}^2 \right)^{1/2}} ,
\end{align*}
where $q < p$. Let $\varphi_{n, p}( \cdot )$ and $\tilde{\varphi}_{n, q}( \cdot )$ denote the characteristic functions of $U_{n, p}$ and $\tilde{U}_{n, q}$, respectively. Then, using inequality ($26.4_0$) in \cite[p.~343]{billingsley2008probability}, we have
\begin{align}
\left| \varphi_{n, p}( t ) - \tilde{\varphi}_{n, q}( t ) \right|
& = \left| \E\left[ \exp\left( \mathrm{i} t U_{n, p} \right) \right] - \E\left[ \exp\left( \mathrm{i} t \tilde{U}_{n, q} \right) \right] \right| \nonumber\\
& = \left| \E\left[ \exp\left( \mathrm{i} t \tilde{U}_{n, q} \right) \exp\left( \mathrm{i} t \left( U_{n, p} - \tilde{U}_{n, q} \right) \right) \right] - \E\left[ \exp\left( \mathrm{i} t \tilde{U}_{n, q} \right) \right] \right| \nonumber\\
& \le \E\left[ \left| \exp\left( \mathrm{i} t \tilde{U}_{n, q} \right) \right| \left| \exp\left( \mathrm{i} t \left( U_{n, p} - \tilde{U}_{n, q} \right) \right) - 1 \right| \right] \nonumber\\
& = \E\left[ \left| \exp\left( \mathrm{i} t \left( U_{n, p} - \tilde{U}_{n, q} \right) \right) - 1 \right| \right] \nonumber\\
& \le | t | \E\left[ \left| U_{n, p} - \tilde{U}_{n, q} \right| \right] ,
\label{lemma_gaussian_eq1}
\end{align}
and \eqref{lemma_gaussian_eq1} yields
\begin{align}
\left| \varphi_{n, p}( t ) - \tilde{\varphi}_{n, q}( t ) \right| 
& \le | t | \left[ \E\left\{ \left( U_{n, p} - \tilde{U}_{n, q} \right)^2 \right\} \right]^{1/2} \nonumber\\
& = | t | \left[ \E\left\{ \left( \frac{\sum_{k = q + 1}^p \delta_{p, k} \left\{ ( \bar{V}_{p, k} )^2 - 1 \right\}}{\left( 2 \sum_{k = 1}^p \delta_{p, k}^2 \right)^{1/2}} \right)^2 \right\} \right]^{1/2} \nonumber\\
& = | t | \left[ \frac{\sum_{k_1 = q + 1}^p \sum_{k_2 = q + 1}^p \delta_{p, k_1} \delta_{p, k_2} \E\left[ \left\{ ( \bar{V}_{p, k_1} )^2 - 1 \right\} \left\{ ( \bar{V}_{p, k_2} )^2 - 1 \right\} \right]}{2 \sum_{k = 1}^p \delta_{p, k}^2} \right]^{1/2} \nonumber\\
& \le | t | \left[ \frac{\sum_{k_1 = q + 1}^p \sum_{k_2 = q + 1}^p \delta_{p, k_1} \delta_{p, k_2} \sqrt{\Var\left\{ ( \bar{V}_{p, k_1} )^2 \right\}} \sqrt{\Var\left\{ ( \bar{V}_{p, k_2} )^2 \right\}}}{2 \sum_{k = 1}^p \delta_{p, k}^2} \right]^{1/2} \nonumber\\
& = | t | \left[ \frac{\left[ \sum_{k = q + 1}^p \delta_{p, k} \sqrt{\Var\left\{ ( \bar{V}_{p, k} )^2 \right\}} \right]^2}{2 \sum_{k = 1}^p \delta_{p, k}^2} \right]^{1/2} .
\label{lemma_gaussian_eq2}
\end{align}
From (A.1) in \cite{zhang2020simple}, we get
\begin{align}
\Var\left\{ ( \bar{V}_{p, k} )^2 \right\}
\le M / n
\label{lemma_gaussian_eq3}
\end{align}
for some positive constant $M$ and for every $p$ and $k$. Then, from \eqref{lemma_gaussian_eq2} and \eqref{lemma_gaussian_eq3}, we get
\begin{align}
\left| \varphi_{n, p}( t ) - \tilde{\varphi}_{n, q}( t ) \right|
\le | t | \frac{\sqrt{M}}{\sqrt{n}} \frac{\sum_{k = q + 1}^p \delta_{p, k} }{\left( 2 \sum_{k = 1}^p \delta_{p, k}^2 \right)^{1/2}} .
\label{lemma_gaussian_eq4}
\end{align}
Since by assumption, $\left( \sum_{k = 1}^p \delta_{p, k} \right) / \left( 2 \sum_{k = 1}^p \delta_{p, k}^2 \right)^{1/2}$ is bounded as $p \to \infty$, we have for $q \ge Q_1$ and $p > q$,
\begin{align}
\left| \varphi_{n, p}( t ) - \tilde{\varphi}_{n, q}( t ) \right|
\le \epsilon
\label{lemma_gaussian_eq5}
\end{align}
for all $n \ge N_1$. Now, for any fixed $q$ and $p$, we get from the central limit theorem and the continuous mapping theorem that
\begin{align*}
\tilde{U}_{n, q} \Longrightarrow \frac{\sum_{k = 1}^q \delta_{p, k} \left\{ W_k - 1 \right\}}{\left( 2 \sum_{k = 1}^p \delta_{p, k}^2 \right)^{1/2}}
\end{align*}
as $n \to \infty$, where $W_k$ are independent chi square random variables with degree of freedom 1. Let $\hat{\varphi}_{p, q}( \cdot )$ denote the characteristic function of $\sum_{k = 1}^q \delta_{p, k} \left\{ W_k - 1 \right\} / \left( 2 \sum_{k = 1}^p \delta_{p, k}^2 \right)^{1/2}$. Then, for all $n \ge N_2$ which depends on $Q_1$, we have
\begin{align}
\left| \tilde{\varphi}_{n, q}( t ) - \hat{\varphi}_{p, q}( t ) \right|
\le \epsilon .
\label{lemma_gaussian_eq6}
\end{align}
Next, let $\hat{\varphi}_{p}( \cdot )$ denote the characteristic function of $\sum_{k = 1}^p \delta_{p, k} \left\{ W_k - 1 \right\} / \left( 2 \sum_{k = 1}^p \delta_{p, k}^2 \right)^{1/2}$. Then, using arguments similar to those used to establish \eqref{lemma_gaussian_eq5}, we get that for all $n \ge N_1$,
\begin{align}
\left| \hat{\varphi}_{p, q}( t ) - \hat{\varphi}_{p}( t ) \right|
\le \epsilon .
\label{lemma_gaussian_eq7}
\end{align}
Combining \eqref{lemma_gaussian_eq5}, \eqref{lemma_gaussian_eq6} and \eqref{lemma_gaussian_eq7}, we get
\begin{align}
\left| \varphi_{n, p}( t ) - \hat{\varphi}_{p}( t ) \right|
\le 3 \epsilon
\end{align}
for all $p > Q_1$ and all $n \ge \max\{ N_1, N_2 \}$. Now, using the central limit theorem and the continuous mapping theorem, we get that for every fixed $q$, there is $N_q$ such that
\begin{align}
\left| \varphi_{n, q}( t ) - \hat{\varphi}_{q}( t ) \right|
\le \epsilon
\end{align}
for every $n \ge N_q$. Taking $n \ge \max\{\max\{ N_q : q = 1, \ldots, Q_1\}, N_1, N_2 \}$, we get that
\begin{align*}
\sup_p\left| \varphi_{n, p}( t ) - \hat{\varphi}_{p}( t ) \right|
\le \epsilon ,
\end{align*}
which completes the proof.
\end{proof}

% Proof of \autoref{thm:1} is placed in the Supplement.
\begin{proof}[Proof of \autoref{thm:1}]
Define $ \tilde{\bh}_p\left( \bX_p, \bY_p \right)
= \bh_p\left( \bX_p, \bY_p \right) - \E\left[ \bh_p\left( \bX_p, \bY_p \right) \right] $, where $ \bX_p $ and $ \bY_p $ are independent $p$-dimensional random vectors.
Define
\begin{align*}
\bH_{p, k}
& = \frac{1}{n} \sum_{l=1}^{K} \frac{n_l}{n_k} \sum_{i_k = 1}^{n_k} \E\left[ \tilde{\bh}_p\left( \bY_{p, k, i_k} , \bY_{p, l, 1} \right) \mid \bY_{p, k, i_k} \right]
+ \frac{1}{n} \sum_{l=1}^{K} \sum_{i_l = 1}^{n_l} \E\left[ \tilde{\bh}_p\left( \bY_{p, k, 1} , \bY_{p, l, i_l} \right) \mid \bY_{p, l, i_l} \right] \\
\intertext{and}
\bG_{p, k, l}
& = \frac{1}{n_k n_l} \sum_{i_k = 1}^{n_k} \sum_{i_l = 1}^{n_l} \tilde{\bh}_p\left( \bY_{p, k, i_k} , \bY_{p, l, i_l} \right) 
- \frac{1}{n_k} \sum_{i_k = 1}^{n_k} \E\left[ \tilde{\bh}_p\left( \bY_{p, k, i_k} , \bY_{p, l, 1} \right) \mid \bY_{p, k, i_k} \right] \\
& \quad - \frac{1}{n_l} \sum_{i_l = 1}^{n_l} \E\left[ \tilde{\bh}_p\left( \bY_{p, k, 1} , \bY_{p, l, i_l} \right) \mid \bY_{p, l, i_l} \right] ,
\end{align*}
where $ k, l = 1, \ldots, K $.
Then,
\begin{align*}
\bar{\bR}_{p, k} - \E\left[ \bar{\bR}_{p, k} \right]
= \bH_{p, k} + n^{-1} \sum_{l=1}^{K} n_l \bG_{p, k, l} .
\end{align*}
%%%%%%%%%%%%%%%%
Next, define
\begin{align*}
& \bW_{n, p} = \sqrt{n} \left( \sqrt{\lambda_{p, 1}} \bH_{p, 1}, \ldots, \sqrt{\lambda_{p, K}} \bH_{p, K} \right) , \text{ and}\\
& \bar{\bh}_p\left( \bX_p, \bY_p \right)
= \tilde{\bh}_p\left( \bX_p, \bY_p \right) 
- \E\left[ \tilde{\bh}_p\left( \bX_p, \bY_p \right) \mid \bX_p \right]
- \E\left[ \tilde{\bh}_p\left( \bX_p, \bY_p \right) \mid \bY_p \right] .
\end{align*}
We have
\begin{align}
& \E\left[ \bar{\bh}_p\left( \bX_p, \bY_p \right) \mid \bX_p \right]
= \E\left[ \bar{\bh}_p\left( \bX_p, \bY_p \right) \mid \bY_p \right]
= \E\left[ \bar{\bh}_p\left( \bX_p, \bY_p \right) \right]
= \mathbf{0}_p .
\label{thm1:eqcond}
%\\
%& \text{and } \left\| \bar{\bh}\left( \bX , \bY \right) \right\| 
%\le 6 .
%\label{thm1:eqbound}
\end{align}
It follows from \eqref{thm1:eqcond} that for every case other than $ ( k, i_k, l_1, i_{l_1} ) = ( k, j_k, l_2, j_{l_2} ) $ and $ ( k, i_k, l_1, i_{l_1} ) = ( l_2, j_{l_2}, k, j_k ) $, we have
\begin{align}
& \E\left[ \left\langle \bar{\bh}_p\left( \bY_{p, k, i_k}, \bY_{p, l_1, i_{l_1}} \right), \; \bar{\bh}_p\left( \bY_{p, k, j_k}, \bY_{p, l_2, j_{l_2}} \right) \right\rangle \right] = 0 .
\label{thm1:eqinner}
\end{align}
Note that $ \bG_{p, k, l} = 
( n_k n_l )^{-1} \sum_{i_k = 1}^{n_k} \sum_{i_l = 1}^{n_l} \bar{\bh}_p\left( \bY_{p, k, i_k} , \bY_{p, l, i_l} \right) $.
Using \eqref{thm1:eqinner}, we have
\begin{align}
& \E\left[ \left\| \sqrt{n_k} \frac{1}{n} \sum_{l=1}^{K} n_l \bG_{p, k, l} \right\|^2 \right] \nonumber\\
& = n_k \sum_{l_1 = 1}^{K} \sum_{l_2 = 1}^{K} \frac{n_{l_1}}{n} \frac{n_{l_2}}{n} \E\left[ \left\langle \bG_{p, k, l_1}, \bG_{p, k, l_2} \right\rangle \right] \nonumber\\
& = \frac{1}{n^2} \frac{1}{n_k} \sum_{i_k = 1}^{n_k} \sum_{j_k = 1}^{n_k} \sum_{l_1 = 1}^{K} \sum_{l_2 = 1}^{K} \sum_{i_{l_1} = 1}^{n_{l_1}} \sum_{j_{l_2} = 1}^{n_{l_2}} \E\left[ \left\langle \bar{\bh}_p\left( \bY_{p, k, i_k} , \bY_{p, l_1, i_{l_1}} \right),
\bar{\bh}_p\left( \bY_{p, k, j_k} , \bY_{p, l_2, j_{l_2}} \right) \right\rangle \right] \nonumber\\
& = \frac{1}{n^2} \frac{1}{n_k} \sum_{i_k = 1}^{n_k} \sum_{l = 1}^{K} \sum_{i_{l} = 1}^{n_{l}} \E\left[ \left\| \bar{\bh}_p\left( \bY_{p, k, i_k} , \bY_{p, l, i_{l}} \right) \right\|^2 \right] \nonumber\\
& \quad + \frac{1}{n^2} \frac{1}{n_k} \sum_{i_k = 1}^{n_k} \sum_{j_k = 1}^{n_k} \E\left[ \left\langle \bar{\bh}_p\left( \bY_{p, k, i_k} , \bY_{p, k, j_k} \right),
\bar{\bh}_p\left( \bY_{p, k, j_k} , \bY_{p, k, i_k} \right) \right\rangle \right] .
\label{thm1:negliexpe}
\end{align}
From \eqref{thm1:eqinner}, we get
\begin{align}
& \sup_p\text{E}\left[ \left\| \sqrt{n_k} \frac{1}{n} \sum_{l=1}^{K} n_l \bG_{p, k, l} \right\|^2 \right]
\longrightarrow 0
\, \text{ as } n \to \infty \text{  for } k = 1, \ldots, K,
\nonumber\\
& \text{which implies }
\sum_{k = 1}^{K} n_k \left\| \left\{ \bar{\bR}_{p, k} - \E\left[ \bar{\bR}_{p, k} \right] \right\} - \bH_{p, k} \right\|^2
= \sum_{k = 1}^{K} \left\| \sqrt{n_k} \frac{1}{n} \sum_{l=1}^{K} n_l \bG_{p, k, l} \right\|^2
\stackrel{P}{\longrightarrow} 0
\label{thm1:eq4}
\end{align}
uniformly-over-$p$ as $ n \to \infty $.
Next, it can be verified that
\begin{align}
\text{Cov}\left( \sqrt{n_{k_1}} \bH_{p, k_1},\, \sqrt{n_{k_2}} \bH_{p, k_2} \right) 
& = \frac{\sqrt{n_{k_1} n_{k_2}}}{n} \sum_{l=1}^{K} \frac{n_l}{n} \left[ \mathbf{C}_p( k_1, k_2, l ) - \mathbf{C}_p( k_1, l, k_2 ) - \mathbf{C}_p( l, k_2, k_1 ) \right] \nonumber\\
& \quad
+ \sum_{l_1 = 1}^{K} \sum_{l_2 = 1}^{K} \frac{n_{l_1} n_{l_2}}{n^2}
\mathbf{C}_p( l_1, l_2, k_1 ) \mathbb{I}( k_1 = k_2 ) .
\label{thm1:eq7}
\end{align}
Since $ n^{-1} n_k \to \lambda_k \in ( 0, 1 ) $ for all $ k $ as $ n \to \infty $, from \eqref{thm1:eq7}, we have for all $ k_1 $ and $ k_2 $,
\begin{align}
& \text{Cov}\left( \sqrt{n_{k_1}} \bH_{p, k_1},\, \sqrt{n_{k_2}} \bH_{p, k_2} \right)
\longrightarrow \boldsymbol{\sigma}_{p, k_1, k_2} \quad \text{as } n \to \infty .
\label{thm1:eq8}
\end{align}
The proof follows from \autoref{lemma_gaussian} using \eqref{thm1:eq8}.
%Since $ E\left[ \mathbf{W}_n \right] = \mathbf{0} $,
%from \eqref{thm1:eq8} and an application of Theorem 1.1 in \cite{kundu2000central}, we get
%$ \mathbf{W}_n \stackrel{w}{\longrightarrow} G\left( \mathbf{0}, \bSigma \right) $ as $ n \to \infty $, which, along with \eqref{thm1:eq4}, implies that
%$ \left[ \mathbf{U}_n - E\left[ \mathbf{U}_n \right] \right]
%\stackrel{w}{\longrightarrow} G\left( \mathbf{0}, \bSigma \right) $ as $ n \to \infty $.

%When $ \text{H}_0 $ in \eqref{h_0} is true, $ \bX_{k 1} $ is an independent copy of $ \bX_{l 1} $ for $ k \neq l $, which, along with the fact that $ \bh( \mathbf{0} ) = \mathbf{0} $ by definition, gives $ E[ \bh( \bX_{k 1} - \bX_{l 1} ) ] = \mathbf{0} $ for all $ k, l $. So, $ E\left[ \sqrt{n_k} \bar{\mathbf{R}}_k \right]
%= \sqrt{n_k} \sum_{l=1}^{K} n^{-1} n_l E\left[ \bh\left( \bX_{k 1} - \bX_{l 1} \right) \right] = \mathbf{0} $ for all $ k $, and consequently, $ E\left[ \mathbf{U}_n \right] = \mathbf{0} $. Therefore from \autoref{thm:1}, we have $ \mathbf{U}_n \stackrel{w}{\longrightarrow} G\left( \mathbf{0}, \bSigma \right) $ as $ n \to \infty $. Hence, from an application of the mapping theorem \cite[Theorem 2.7, p.~21]{billingsley2013convergence}, we get
%$ \| \mathbf{U}_n \|^2 \stackrel{w}{\longrightarrow} \| \mathbf{W} \|^2 $ as $ n \to \infty $, where $ \mathbf{W} $ is a random element having distribution $ G\left( \mathbf{0}, \bSigma \right) $.
\end{proof}

\begin{proof}[Proof of \autoref{thm:2}]
Under $ \text{H}_0 $ in \eqref{h_0}, it follows from \autoref{thm:1} that for every $ 0 < \alpha < 1 $, the size of the level $ \alpha $ kernel-based test converges to $ \alpha $ as $ n \to \infty $ uniformly-over-$p$.

Since for all $ k $, $ E[ \bar{\mathbf{R}}_{p,k} ] \to \sum_{l=1}^{K} \lambda_l E[ \mathbf{h}_p( \bY_{p, k, 1} , \bY_{p, l, 1} ) ] $ as $ n \to \infty $ uniformly-over-$p$, if there is $k$ such that $ \sum_{l=1}^{K} \lambda_l E[ \mathbf{h}_p( \bY_{p, k, 1} , \bY_{p, l, 1} ) ] \neq \mathbf{0}_p $ for every $p$, we have $ \left\| E\left[ \sqrt{n_k} \bar{\mathbf{R}}_{p, k} \right] \right\| \to \infty $ as $ n \to \infty $ uniformly-over-$p$ for that $ k $, and this implies $ S_{n,p} \stackrel{P}{\longrightarrow} \infty $ as $ n \to \infty $ uniformly-over-$p$. So, the power of the test converges to 1 as $n \to \infty$ uniformly-over-$p$.
\end{proof}

\end{appendix}
\bibliographystyle{imsart-number} % Style BST file (imsart-number.bst or imsart-nameyear.bst)
\bibliography{bibliography}       % Bibliography file (usually '*.bib')

%% or include bibliography directly:
% \begin{thebibliography}{}
% \bibitem{b1}
% \end{thebibliography}

\end{document}